\documentclass[10pt]{amsart}
\usepackage{etoolbox}
\usepackage{dynkin-diagrams}
\def\row#1/#2!{#1_{\IfStrEq{#2}{}{n}{#2}} & \dynkin{#1}{#2}\\}

\usepackage{accents}
\usepackage{graphicx}
\usepackage{mathrsfs}
\graphicspath{ {./images/} }
\newcommand{\congto}{\xrightarrow{\raisebox{-1ex}[0ex][0ex]{$\sim$}}}
\usepackage{amsmath}
\usepackage{enumerate}
\usepackage{mathrsfs}
\usepackage{multirow}
\usepackage[inline]{enumitem}
\usepackage{tikz-cd}
\usepackage{tikz}
\usepackage{graphicx}
\usepackage{float}
\usepackage{stmaryrd}
\usepackage{blkarray}

\newcounter{braid}
\newcounter{strands}

\pgfkeyssetvalue{/tikz/braid height}{1cm}
\pgfkeyssetvalue{/tikz/braid width}{1cm}
\pgfkeyssetvalue{/tikz/braid start}{(0,0)}
\pgfkeyssetvalue{/tikz/braid colour}{black}
\pgfkeys{/tikz/strands/.code={\setcounter{strands}{#1}}}

\def\cross{%
  \@ifnextchar^{\message{Got sup}\cross@sup}{\cross@sub}}

\def\cross@sup^#1_#2{\render@cross{#2}{#1}}

\def\cross@sub_#1{\@ifnextchar^{\cross@@sub{#1}}{\render@cross{#1}{1}}}

\def\cross@@sub#1^#2{\render@cross{#1}{#2}}

\def\render@cross#1#2{
  \def\strand{#1}
  \def\crossing{#2}
  \pgfmathsetmacro{\cross@y}{-\value{braid}*\braid@h}
  \pgfmathtruncatemacro{\nextstrand}{#1+1}
  \foreach \thread in {1,...,\value{strands}}
  {
    \pgfmathsetmacro{\strand@x}{\thread * \braid@w}
    \ifnum\thread=\strand
    \pgfmathsetmacro{\over@x}{\strand * \braid@w + .5*(1 - \crossing) * \braid@w}
    \pgfmathsetmacro{\under@x}{\strand * \braid@w + .5*(1 + \crossing) * \braid@w}
    \draw[braid] \pgfkeysvalueof{/tikz/braid start} +(\under@x pt,\cross@y pt) to[out=-90,in=90] +(\over@x pt,\cross@y pt -\braid@h);
    \draw[braid] \pgfkeysvalueof{/tikz/braid start} +(\over@x pt,\cross@y pt) to[out=-90,in=90] +(\under@x pt,\cross@y pt -\braid@h);
    \else
    \ifnum\thread=\nextstrand
    \else
     \draw[braid] \pgfkeysvalueof{/tikz/braid start} ++(\strand@x pt,\cross@y pt) -- ++(0,-\braid@h);
    \fi
   \fi
  }
  \stepcounter{braid}
}

\tikzset{braid/.style={double=\pgfkeysvalueof{/tikz/braid colour},double distance=1pt,line width=2pt,white}}

\newcommand{\braid}[2][]{%
  \begingroup
  \pgfkeys{/tikz/strands=2}
  \tikzset{#1}
  \pgfkeysgetvalue{/tikz/braid width}{\braid@w}
  \pgfkeysgetvalue{/tikz/braid height}{\braid@h}
  \setcounter{braid}{0}
  \let\sigma=\cross
  #2
  \endgroup
}
\makeatother

\definecolor{airforceblue}{rgb}{0.36, 0.54, 0.66}
\definecolor{babyblue}{rgb}{0.54, 0.81, 0.94}
\definecolor{bleudefrance}{rgb}{0.19, 0.55, 0.91}
\definecolor{darkpink}{rgb}{0.91, 0.33, 0.5}
\usepackage{commutative-diagrams}
\usetikzlibrary{commutative-diagrams}
\usepackage{mathtools}
\usepackage{stix}

\DeclarePairedDelimiterX\braket[2]{\langle}{\rangle}{#1 \delimsize\vert #2}
\usepackage{dynkin-diagrams}
\usetikzlibrary{matrix,arrows,decorations.pathmorphing}
\usepackage{setspace}

\usepackage[
  maxcitenames=20,
  minbibnames=30, maxbibnames=30,
]{biblatex}

\usepackage{amsthm}
\usepackage{thmtools}
\usepackage{CJKutf8}
\usepackage[CJKbookmarks=true,
bookmarksnumbered=true,
bookmarksopen=true,
colorlinks,
linkcolor=bleudefrance,
anchorcolor=green,
citecolor=darkpink
]{hyperref}

\def\det{\mathrm{det}}

\def\det{\mathrm{det}}

\def\C{\mathbb C}

\def\Z{\mathbb Z}

\def\N{\mathbb N}

\renewbibmacro{in:}{%
  \ifentrytype{article}
    {}
    {\bibstring{in}%
     \printunit{\intitlepunct}}}

\theoremstyle{plain}
\newtheorem{thm}{Theorem}[section]%
\newtheorem{prop}[thm]{Proposition}%
\newtheorem{lma}[thm]{Lemma}%

\theoremstyle{definition}
\newtheorem{define}[thm]{Definition}%

\newtheorem{eg}[thm]{Example}%
\newtheorem{rmk}[thm]{Remark}
\begin{document}
\title{Integrable system on minimal nilpotent orbit}
\author{Xinyue Tu}
\address{Bâtiment 307, rue Michel Magat,
Faculté des Sciences d’Orsay, Université Paris Saclay}
\email{logic520@163.com}
\address{ORCiD: 0009-0007-2211-1456}

\begin{abstract}
We show that for every complex simple Lie algebra $\mathfrak{g}$, the equations of Schubert divisors on the flag variety $G/B^-$ give a complete integrable system of the minimal nilpotent orbit $\mathcal{O}_{\min}$. The approach is motivated by the integrable system on Coulomb branch \cite{bfn16}. We give explicit computations of these Hamiltonian functions, using Chevalley basis and a so-called Heisenberg algebra basis.  For classical Lie algebras we rediscover the lower order terms of the celebrated Gelfand-Zeitlin system. For exceptional types we computed the number of Hamiltonian functions associated to each vertex of Dynkin diagram. They should be regarded as analogs of Gelfand-Zeitlin functions on exceptional type Lie algebras.
\end{abstract}
\maketitle
\makeatletter
\def\keywords{\xdef\@thefnmark{}\@footnotetext}
\makeatother

\pagenumbering{Roman}
\tableofcontents

\pagenumbering{arabic}
\nocite{*}

\section{Introduction}\label{miaowu}

A. Braverman, M. Finkelberg and H. Nakajima in their series of papers (\cite{bfn16},  \cite{bfn162}) proposed a mathematically rigorous description of a Coulomb branch $\mathcal{M}_C(G, \mathbf{N})$. It is a certain hyper-Kähler manifold that depends on a complex reductive algebraic group $G$ and its complex representation $\mathbf{N}$. Let $\mathfrak{t}$ be the Lie algebra of a maximal torus of $G$ and $W$ be the Weyl group. Then $\mathbb{C}[\mathfrak{t} / W]$ is a Poisson commutative subalgebra of $\mathbb{C}\left[\mathcal{M}_C(G, \mathbf{N})\right]$; therefore, a Coulomb branch can be equipped with an integrable system (see remark 3.11 in \cite{bfn16}  and lemma 5.11 in \cite{bfn162}) .
A special case of interest is when $\mathcal{M}_C(G, \mathbf{N})$ arises from a quiver gauge theory. In that case, the Coulomb branch can be identified with a generalized Grassmannian slice $\overline{\mathcal{W}} ^\lambda_{\mu}$, which depends on coweights $\mu$ and $\lambda$ of $G$. In the special case when $\mu=0$ and $\lambda$ is the quasi-miniscule coweight denoted by $\theta$, $\overline{\mathcal{W}}_0^\theta \cong \overline{\mathcal{O}}_{\text {min }}=\mathcal{O}_{\text {min }} \cup\{0\}$, where $\mathcal{O}_{\text{min}}$ is the minimal nilpotent orbit associated to $\mathfrak{g}$. In this note we show that this isomorphism is Poisson (with the Poisson bracket on $\mathcal{O}_{\text {min }}$ given by the Kirillov-Kostant bracket). As such, the minimal nilpotent orbit $\mathcal{O}_{\min }$ is a Coulomb branch.

We specialize the general construction of integrable systems on $\mathcal{M}_C(G, \mathbf{N})$ to $\mathcal{O}_{\text {min }}$. It is shown that the integrable system is given by the equations of Schubert divisiors on the flag variety $G / B$, where $B$ is a Borel subgroup of $G$. As a final result, the integrable systems are computed explicitly in all Lie types in the Chevalley basis and in some specific representations.

 Besides the minimal nilpotent orbit, there are some other nilpotent orbits which are Coulomb branches. For example, in \cite{achar13} some nilpotent orbits of small dimension can be viewed inside the affine grassmannian, and integrable system of similar form as this note is expected. If one allow for Coulomb branches of orthosymplectic quiver gauge theories, then special nilpotent orbits of classical type are Coulomb branches in this sense, see \cite{hanany}. However, the explicit form of integrable systems on these orbits are not known. 
\subsection{Notations and statement of theorem}

 Throughout this note, let $G$ be a complex simple algebraic group with Lie algebra $\mathfrak{g}$ and rank $l$. Pick a Borel subgroup $B$ inside $G$ with corresponding Lie algebra $\mathfrak{b}$, $U \subset B$ is the unipotent radical, $ U^{-} \subset B^{-}$is the unipotent radical of the opposite Borel, a maximal torus is denoted by $T=B\cap B^-$. 
We denote by $\mathcal{B}$ the flag variety of $G$, thus we have $\mathcal{B}=G/B=G/B^-$.
Let $\Lambda_G$ be the coweight lattice of $G$,  $\Lambda_G^{+} \subset \Lambda_G$ be the dominant part of the coweight lattice,  and $\check{\Lambda}_G^{+} \subset \check{\Lambda}_G$ be the dominant part of the weight lattice. 

\textbf{Basis, weights, parabolic subgroup.} Fix a set of Chevalley generators $\left(e_1, f_1, h_1, \ldots, e_n, f_n, h_n\right)$ of $\mathfrak{g}$ which is compatible with the choice of borel $\mathfrak{b}$. We have $h_i=\left[e_i, f_i\right]$ for $i \in\{1, \ldots, n\}$ and the decomposition $\mathfrak{g}=\mathfrak{n}_{+} \oplus \mathfrak{h} \oplus \mathfrak{n}_{-}$, where $\mathfrak{n}_{+}$is generated by $\left\{e_i \mid 1 \leq i \leq n\right\}, \mathfrak{n}_{-}$is generated by $\left\{f_i\right\}$ and the Cartan subalgebra $\mathfrak{h}$ is spanned by $h_i$. 
Let $\{\varpi_i| 1\leq i\leq l\}$ be the set of fundamental weights (the order $i$ as in Bourbaki's notation \cite{bou}), and $\{\varpi^\vee_i|1\leq i\leq l\}$ be the set of fundamental coweights. Let $P_i$ be the maximal parabolic subgroup associated to $\varpi_i$, and $L_i$ be the Levi factor of $P_i$. Let $W_{P_i}$ be the Weyl group of $P_i$. Let $R^+$ be the positive roots of $\mathfrak{g}$, and $R_{P_i}$ be the roots associate with $P_i$, i.e, $R_{P_i}$ are linear combinations of $\alpha_j$ ($j\neq i$).  Denote the coroot associated with $\alpha$ by $H_\alpha=\alpha^\vee$. For any fundamental weight $\varpi_i$ and simple root $\alpha_j$, we have $\langle \varpi_i, \alpha_j^\vee\rangle=2\frac{\langle\varpi_i,\alpha_j\rangle}{\langle \alpha_j,\alpha_j\rangle}=\delta_{ij}$.  

\textbf{Representations.}
 For each vertex $1\leq k\leq l=\mathrm{rank}(G)$ of the Dynkin diagram of $G$, denote the associated fundamental representation by $V_k$ and highest weight vector by $v_k$. For any $\chi \in \check{\Lambda}_G^{+}$, let $V_\chi$ be the irreducible $G$-module of highest weight $\chi$ and $L_\chi \subset V_\chi$ be the highest weight line. Let $\theta^\vee$ be the highest weight of adjoint representation of $G$, then $\theta$ is the quasi-minuscule coweight. Pick a basis $(e_i)_{i=1,\cdots, l}$ of $\mathfrak{g}$.  We define
    $\mathrm{Cas}:=\sum_{i=1}^{l}e_i\otimes e_i^*\in U(g)\otimes U(g)^*$,
    and  write its power using Sweedler's notation:
$\mathrm{Cas}^n=\mathrm{Cas}^n_{(1)}\otimes\mathrm{Cas}_{(2)}^n$.

 \textbf{Orbit, transversal slice and Zastava.}
Let $\mathcal{O}_{\min}$ be the minimal nilpotent orbit inside $\mathfrak{g}$, i.e, $\mathcal{O}_{\min}=G\cdot e_{\theta^\vee}$, where $e_{\theta^\vee}$ is the root vector of highest weight,  From \cite{wang} we know $\dim\,\mathcal{O}_{\min}=2h^\vee-2$, where $h^\vee$ is the dual coxeter number of $G$. Let $\overline{\mathcal{W}} ^\lambda_{-\mu}$ be the generalized affine grassmanian slice with coweights $\lambda$, $\mu$, we will recall its definition in \ref{2.2}. For any algebraic group $H$, let $H_1[z^{-1}]$ be the kernel of evaluation map $H[z^{-1}] \rightarrow H$ at $z \rightarrow \infty$. Let $\mathring{\mathcal{Z}}^{\alpha}$ be the open Zastava associated with coweight $\alpha$, $\mathcal{Z}^{\alpha}$ be the Zastava space. We will recall them in \ref{wyh}.  Our main theorem is the following:
\begin{thm}[Main theorem]\label{998}
\begin{enumerate}
\item
We have maps $\overline{\mathcal{W}}_{-\theta}^0\cong\mathring{\mathcal{Z}}^\theta\hookrightarrow \mathcal{O}_{\min}\stackrel{\exp}{\rightarrow}Gr_G(\mathbb{C})\hookrightarrow \mathcal{Z}^\theta$ identifying the Kirillov-Kostant Poisson structure on the minimal nilpotent orbit with the Poisson structure of the generalized affine grassmannian  $\overline{\mathcal{W} }^0_{-\theta}$.         

\item
Define the functions 
\[f_{n,k}:=\mathrm{Cas}_{(2)}^n\langle\mathrm{Cas}_{(1)}^n v_k, v^k\rangle (x)\in\C[\mathfrak{g}] \quad\text{for}\, 1\leq k\leq l, \,n\geq1.\]
Then for each vertex  $k$, there is a number $m_k$  such that $\{f_{n,k}|1\leq k\leq l, 1\leq n\leq m_k\}$ form a complete integrable system on $\mathcal{O}_{\min}$, and they are induced from the equations of Schubert divisors on opposite flag $G/B^-$.  If $G$ is of classical type, then these functions are lower terms of Gelfand-Zeitlin's Hamiltonians.   

We list the number $m_k$ associated to each $k$ of every $G$ as in Figure 1, note that the numbers on vertexes of Dynkin diagram add up to $h^\vee-1$, which is half dimension of the minimal nilpotent orbit. 
\begin{figure}\label{xiaoxiao}
\centering
\begin{align*}
A_n &&& \underset{\mathclap{1}}{\circ} - \underset{\mathclap{1}}{\circ} - \dotsb - \underset{\mathclap{1}}{\circ} - \underset{\mathclap{1}}{\circ} && h^\vee-1=n \\
B_n &&& \underset{\mathclap{1}}{\circ} - \underset{\mathclap{2}}{\circ} - \dotsb - \underset{\mathclap{2}}{\circ} \Rightarrow \underset{\mathclap{1}}{\circ} && h^\vee-1=2n-2 \\
C_n &&& \underset{\mathclap{1}}{\circ} - \underset{\mathclap{1}}{\circ} - \dotsb - \underset{\mathclap{1}}{\circ} \Leftarrow \underset{\mathclap{1}}{\circ} && h^\vee-1= n\\
D_n &&& \underset{\mathclap{1}}{\circ} - \underset{\mathclap{2}}{\circ} - \dotsb - \underset{\mathclap{2}}{\overset{\overset{\textstyle\circ_{\mathrlap{1}}}{\textstyle\vert}}{\circ}} \,-\, \underset{\mathclap{1}}{\circ} && h^\vee-1=2n-3 \\
E_6 &&& \underset{\mathclap{1}}{\circ} - \underset{\mathclap{2}}{\circ} - \underset{\mathclap{3}}{\overset{\overset{\textstyle\circ_{\mathrlap{2}}}{\textstyle\vert}}{\circ}} - \underset{\mathclap{2}}{\circ} - \underset{\mathclap{1}}{\circ} && h^\vee-1=11 \\
E_7 &&& \underset{\mathclap{2}}{\circ} - \underset{\mathclap{3}}{\circ} - \underset{\mathclap{4}}{\overset{\overset{\textstyle\circ_{\mathrlap{2}}}{\textstyle\vert}}{\circ}} - \underset{\mathclap{3}}{\circ} - \underset{\mathclap{2}}{\circ} - \underset{\mathclap{1}}{\circ} && h^\vee-1=17\\
E_8 &&& \underset{\mathclap{2}}{\circ} - \underset{\mathclap{3}}{\circ} - \underset{\mathclap{4}}{\circ} - \underset{\mathclap{5}}{\circ} - \underset{\mathclap{6}}{\overset{\overset{\textstyle\circ_{\mathrlap{3}}}{\textstyle\vert}}{\circ}} - \underset{\mathclap{4}}{\circ} - \underset{\mathclap{2}}{\circ} && h^\vee-1=29 \\
F_4 &&& \underset{\mathclap{2}}{\circ} - \underset{\mathclap{3}}{\circ} \Rightarrow \underset{\mathclap{2}}{\circ} - \underset{\mathclap{1}}{\circ} && h^\vee-1=8 \\
G_2 &&& \underset{\mathclap{2}}{\circ} \Rrightarrow \underset{\mathclap{1}}{\circ} && h^\vee-1=3
\end{align*}
\caption{}
\end{figure}
\end{enumerate}
\end{thm}

\subsection{Strategy of proof}

The desired geometric integrable system come from the fact that the minimal nilpotent orbit is an example of the Coulomb branch.  From remark 3.11 in \cite{bfn16}  and lemma 5.11 in \cite{bfn162} , one can construct a Poisson morphism from the open Zastava $\mathring{\mathcal{Z}}^\theta$ to the minimal nilpotent orbit via Coulomb branch description.  There is an integrable system on open Zastava 
(see diagram \eqref{miao}). However, to carry the integrable system on  the open Zastava (identified as an affine grassmanian slice) to minimal nilpotent orbit and give concrete formula, one has to verify that the exponential map $\exp:\,\overline{\mathcal{O}_{min}}\to \mathrm{Gr}_G(\C)$ is Poisson, with the Poisson structure on the nilpotent orbit side be the Kirillov-Kostant strucuture, and on the affine grassmanian slice side is the Poisson structure induce by the Manin triple $\mathfrak{g}((t^{-1})),t^{-1}\mathfrak{g}[[t^{-1}]],\mathfrak{g}[t])$.  \par
To show our exponential map is Poisson, the key step is to observe that the exponential map is $G$-equivariant (which is a proposition by Beilinson-Drinfeld \cite[pp.183]{bei91}). This is shown in proposition \ref{2.10}.
Finally, investigating the modular definition of the open Zastava and its integrable system $\pi$ yields that its Hamiltonians are essentially the same as equations of Schubert divisors of flag variety $G/B^-$, as in lemma \ref{2.9}.

Then we can compute the expression of Hamiltonians in Chevalley basis \eqref{33}. This method gives infinite hamitonian functions at first glance, but  from \cite{wang} we know $\dim\,\mathcal{O}_{\min}=2h^\vee-2$, so there should be $h^\vee-1$ (i.e, half-dimension) independent Poisson-commuting Hamiltonian functions, and it is an essential issue to figure out if the Hamiltonian functions vanish when restricted from $\mathfrak{g}$ to $\mathcal{O}_{\min}$.

A possible way is to use Joseph's ideal $\mathcal{J}$ \cite{Jos} and check that the higher-order Hamiltonians all fall in the commutative algebra $\operatorname{gr}\mathcal{J}$ (which is the vanishing ideal of $\mathcal{O}_{\min}$ in $\mathfrak{g}$), but the generators of $\mathcal{J}$ are extremely complicated, so we would not take this approach. 
 
One can determine which Hamiltonians are non-vanishing once we switch to a Heisenberg algebra's basis. The higher order Hamiltonians vanish due to an analysis on the weight diagram for each fundamental weight representation and we obtain the numbers $m_k$  in section \ref{ppl}.  They satisfy $\sum_k m_k=h^\vee-1=\frac{1}{2}\dim \mathcal{O}_{\min}$, so the Hamiltonians above form a complete integrable system.
We also remark that these numbers $m_k$ bear a similarity with the coordinates of the unique vector $\delta=(a_0,a_1,\cdots,a_l)$ such that $A\delta=0$ and the $a_i$ are positive relatively prime integers, where $A$ is the generalized Cartan matrix of $G$, see in \cite[pp.44]{kac}.

\subsection{Organization of the paper:}In the remaining part of section \ref{miaowu}, we briefly recall the origin of integrable systems on Coulomb branch.

In section \ref{ffd}, we show the integrable system on the open Zastava (which is a special Coulomb branch, with Poisson structure coming from a Manin triple ) can be translated into integrable system on minimal nilpotent orbit, with the Kirillov-Kostant Poisson structure .

In section \ref{s4}, we give general formulas for the Hamiltonians for every complex simple Lie algebra using two basis: Chevelley basis and Heisenberg basis. The order-1 Hamiltonians are projections to Cartan, and order-2 Hamiltonians also have simple formula \ref{2222}. 

In section \ref{1111}, 
 we compute the Hamiltonians for the integrable system on minimal nilpotent for classical groups in terms of matrix elements basis, and obtain that they conincide with the lower terms of Gelfand Zeitlin systems. Then we conjecture the quantization of the Hamiltonians.

\subsection{Background on integrable system of Coulomb branch}

Braverman, Finkelberg and Nakajima \cite{bfn16} have given a mathematical definition of the Coulomb branch of a 3d N=4 gauge theory of cotangent type. These Coulomb branches are affine Poisson varieties that come equipped with a canonical integrable system and natural quantization of their ring of functions.

Let $\mathbf{N}$ be a finite dimensional representation of a complex simple algebraic group $G$. We consider the moduli space $\mathcal{R}_{G, \mathbf{N}}$ of triples $(\mathcal{P}, \sigma, s)$ where $\mathcal{P}$ is a $G$-bundle on the formal $\operatorname{disc} D=\operatorname{Spec} \mathcal{O}=\operatorname{Spec}\mathbb{C}[\![t]\!] ; \sigma$ is a trivialization of $\mathcal{P}$ on the punctured formal disc $D^*=\operatorname{Spec} \mathcal{K}=\operatorname{Spec}\mathbb{C}(\!(t)\!)$; and $s$ is a section of the associated vector bundle $\mathcal{P}_{\text {triv }}\times^G \mathbf{N}$ on $D^*$ such that $s$ extends to a regular section of $\mathcal{P}_{\text {triv }}\times^G\mathbf{N}$ on $D$, and $\sigma(s)$ extends to a regular section of $\mathcal{P} \times \mathbf{N}$ on $D$. In other words, $s$ extends to a regular section of the vector bundle associated to the $G$-bundle glued from $\mathcal{P}$ and $\mathcal{P}_{\text {triv }}$ on the non-separated formal scheme glued from 2 copies of $D$ along $D^*$ (raviolo). The group $G_{\mathcal{O}}$ acts on $\mathcal{R}_{G, \mathbf{N}}$ by changing the trivialization $\sigma$.

The Coulomb branch of BFN for the pair $(G,\mathbf{N})$ (also called for the Gauge theory $\mathcal{T}(T^*\mathbf{N})/G$)  is defined to be 
$\mathcal{M}_C(G, \mathbf{N}):=\mathrm{Spec} \,H_{\bullet}^{G_{\mathcal{O}}}(\mathcal{R}_{G, \mathbf{N}})$, they are discussed in detail in \cite{bfn16}, \cite{bfn162}.  The equivariant cohomology $H^{G_{\mathcal{O}}}_{\bullet}(\mathrm{pt})=\mathbb{C}[\mathfrak{t} / W]$ being a subalgebra of $H_{\bullet}^{G_{\mathcal{O}}}(\mathcal{R}_{G, \mathbf{N}}$ ) (actually is a Cartan subalgebra) yields the desired integrable system \[\pi:\,\mathcal{M}_C(G, \mathbf{N})=\mathrm{Spec} \,H_{\bullet}^{G_{\mathcal{O}}}\left(\mathcal{R}_{G, \mathbf{N}}\right)    \rightarrow \mathfrak{t} / W= \mathrm{Spec} \, H^{G_{\mathcal{O}}}_{\bullet}(\mathrm{pt}) . \] The integrability comes from the fact that the algebra $H_{\bullet}^{G_{\mathcal{O}}}\left(\mathcal{R}_{G, \mathbf{N}}\right)$ is equiped with quantization: a  deformation of  $\mathbb{C}[\hbar]-$ associative algebra $\mathbb{C}_{\hbar}\left[\mathcal{M}_C(G, \mathbf{N})\right]:=H_{\bullet}^{G_{\mathcal{O}} \rtimes \mathbb{C}^{\times}}\left(\mathcal{R}_{G, \mathbf{N}}\right)$ where $\mathbb{C}^{\times}$acts by loop rotations, and $\mathbb{C}[\hbar]:=H_{\mathbb{C}^{*}}^{\bullet}(\mathrm{pt})$. It gives rise to a Poisson bracket on $\mathbb{C}\left[\mathcal{M}_C(G, \mathbf{N})\right]$ with an open symplectic leaf, so that $\pi$ becomes an integrable system: $\mathbb{C}[\mathfrak{t} / W] \subset$ $\mathbb{C}\left[\mathcal{M}_C(G, \mathbf{N})\right]$ is a Poisson-commutative polynomial subalgebra with $\operatorname{rk}(G)$ generators.  

The study of this integrable system $\pi$ and finding the Hamiltonians 
is an object of great interest in the study of Gauge theory.
We consider quiver Gauge theory, let $G$ be an simple algebraic group, choose an orientation of its Dynkin graph and obtain a quiver $Q$, with $Q_0$ the set of vertices, and $Q_1$ the set of arrows. An arrow $e \in Q_1$ goes from its tail $t(e) \in Q_0$ to its head $h(e) \in Q_0$. We choose a $Q_0$-graded vector spaces $V:=\bigoplus_{j \in Q_0} V_j$ and $W:=\bigoplus_{j \in Q_0} W_j$. 

We encode the dimension of the $Q_0$-graded space $W$ by a dominant coweight 

$\lambda:=\sum_{i \in I} \operatorname{dim}\left(W_i\right) \omega_i \in \Lambda^{+}_G$, where $\Lambda_G$ is the coweight lattice, and encode the dimension of graded vector space $V$ by 
a positive $\text { coroot combination } \alpha:=\sum_{i \in I} \operatorname{dim}\left(V_i\right) \alpha_i \in \check{\Lambda^{+}_G}$, where $\check{\Lambda_G}$ is the weight lattice. The gauge group is defined to be $\mathbf{G}=\mathrm{GL}(V):=\prod_{j \in Q_0} \mathrm{GL}\left(V_j\right)$. 
\\ $\text { We set } \mu:=\lambda-\alpha \in \Lambda_G $, $\mathbf{N}:=\bigoplus_{e \in Q_1} \operatorname{Hom}\left(V_{t(e)}, V_{h(e)}\right) \oplus \bigoplus_{j \in Q_0} \operatorname{Hom}\left(W_j, V_j\right).$ In this case, the Coulomb branch associated to the pair $(G, \mathbf{N})$ is isomorphic to the generalized affine grassmannian slice:

\begin{equation}\label{222}\overline{\mathcal{W}}_\mu^\lambda \stackrel{\sim}{\longrightarrow} \mathcal{M}_C(\mathbf{G}, \mathbf{N})\end{equation}
Notice that the group used in the Coulomb branch for quiver gauge theory is the gauge group $\mathbf{G}$ rather than the original group $G$. 

\begin{equation}\label{22222}
\centering
\begin{tikzcd}
\mathbb{A}^{\frac{n(n-1)}{2}}& \mathcal{N}_{\mathfrak{gl}_N}\arrow[l,dashrightarrow]\\
\mathbb{A}^{h^\vee-1}\arrow[u,hook]& \overline{\mathcal{O}_{min}}\arrow[l, dashrightarrow, "\pi"]\arrow[u,hook]\\
& \accentset{\circ}{\mathcal{Z}}^\theta\cong \mathcal{M}_C(\mathbf{G}, \mathbf{N})\arrow[u,hook]\arrow[ul,"\pi"]
\end{tikzcd}
\end{equation}

For the definition of generalized affine grassmannian slice we refer to \ref{2.2}, and the original construction goes to \cite{fm}.
The isomorphism \eqref{222} was proved by Braverman-Finkelberg-Nakajima for the $ADE$ type  quiver in \cite{bfn16} , and later in \cite{nw} for $BCFG$ type by Nakajima-
Weekes.  We have
  $\mathcal{M}_C(\mathbf{G}, \mathbf{N})\stackrel{\sim}{\rightarrow} 
 \accentset{\circ}{\mathcal{Z}}^\theta\stackrel{\sim}{\longrightarrow}\overline{\mathcal{W}}_{-\theta}^0 \hookrightarrow \overline{\mathcal{W}}^\theta_0\cong \overline{\mathcal{O}_{min}}$ , where the open embedding holds for general, i.e, $\overline{\mathcal{W}}_{-d}^0 \hookrightarrow \overline{\mathcal{W}}^a_{a-d} $, and the last isomorphism comes from \cite{mal05}.  The diagram \eqref{22222} is drawn in type A ,  our results in the paper is a small step towards understanding integrable system on all nilpotent orbit.

\section{Construction of Integrable Systems}\label{ffd}

\subsection{Generalized transversal slice and Zastava }
We recall the Zastava space following \cite{bfn16} and then describe the integrable system on it.

\begin{define}\label{wyh}
For a dominant coweight $\alpha$, let $\mathring{\mathcal{Z}}^{\alpha}$ be the space of maps $\mathbb{P}^1\to \mathcal{B}$ of degree $\alpha$ sending $\infty\in\mathbb{P}^1$ to $\mathfrak{b}_-\in \mathcal{B}$. This is called the open Zastava space. It is known \cite{fm} that this is a smooth symplectic affine algebraic variety. 
We denote by $\mathcal{Z}^\alpha$ the corresponding space of based quasi-maps of degree $\alpha$ (i.e. quasi-maps of degree $\alpha$ which have no defect at $\infty$ and such that the corresponding map sends $\infty$ to $\mathfrak{b}_-$). This is called Zastava space and is an affine algebraic variety.
\end{define}

\begin{define}
Denote by $K_1$ the first congruence subgroup of $G_{\mathbb{C}\left[\left[z^{-1}\right]\right]}$, which is the kernel of the evaluation projection $\mathrm{ev}_{\infty}: G_{\mathbb{C}\left[\left[z^{-1}\right]\right]} \rightarrow G$. When $\lambda \in \Lambda_G^{+}, \mu \in \Lambda_G^+$, the transversal slice $\mathcal{W}_\mu^\lambda$
(resp. $\overline{\mathcal{W}}_\mu^\lambda$ ) is defined as the intersection of $\operatorname{Gr}_G^\lambda$ (resp. $\overline{\operatorname{Gr}}_G^\lambda$ ) and $K_1 \cdot \mu$ in the affine grassmannian $\mathrm{Gr}_G$. It is known that $\overline{\mathcal{W}}_\mu^\lambda$ is nonempty iff $\mu \leq \lambda$, and $\operatorname{dim} \overline{\mathcal{W}}_\mu^\lambda$ is an affine irreducible variety of dimension $\left\langle 2 \rho^{\vee}, \lambda-\mu\right\rangle$. Following an idea of I. Mirković, \cite{KWY} proved that $\overline{\mathcal{W}}_\mu^\lambda=\bigsqcup_{\mu \leq \nu \leq \lambda} \mathcal{W}_\mu^\nu$ is the decomposition of $\overline{\mathcal{W}}_\mu^\lambda$ into symplectic leaves of a natural Poisson structure.
\end{define}

\begin{define} \label{2.2}
Let $\lambda \in \Lambda_G^{+}$, but with $\mu \in \Lambda_G$ no longer required to be dominant.  We use the same notation to denote the  generalized transversal slice $\mathcal{W}_\mu^\lambda= B_1[t^{-1}]t^\mu U_1^-[t^{-1}]\cap G_{\mathcal{O}}t^\lambda G_{\mathcal{O}}$, its closure $\overline{\mathcal{W}}_\mu^\lambda$ has a modular definition, which is just the parametrization of the equivalence class of data $\left(\mathcal{F}_B, \phi\right)$, where $\mathcal{F}_B$ is a $B$-bundle on $\mathbb{P}^1$ of degree $w_0 \mu$. For any weight $\chi \in \check{\Lambda}_G$,  we have the associated  one dimensional $T$-module $L_\chi$. The induced line bundle $\mathcal{F}_B^{L_\chi}$ has degree $-\left\langle w_0 \mu, \chi\right\rangle$, and $\phi:\left.\left.\operatorname{Ind}_B^G \mathcal{F}_B\right|_{\mathbb{P}^1 \backslash\{0\}} \cong \mathcal{F}_G^{\text {triv }}\right|_{\mathbb{P}^1 \backslash\{0\}}$ is an isomorphism between the induced $G$-bundle $\operatorname{Ind}_B^G \mathcal{F}_B$ with the trivial $G$-bundle $\mathcal{F}_G^{\text {triv }}$ on the open set $\mathbb{P}^1 \backslash\{0\}$, such that the order of pole of $\phi$ is at most $\lambda$, i.e. for every irreducible $G$-module $V_\chi$ of highest weight $\chi \in \check{\Lambda}_G^{+}$, we have inclusions between sheaves
\begin{equation}\label{3}
\mathcal{F}_B^{V_\chi}\left(\left\langle w_0 \lambda, \chi\right\rangle[0]\right) \subset V_\chi \otimes \mathcal{O}_{\mathbb{P}^1} \subset \mathcal{F}_B^{V_\chi}(\langle\lambda, \chi\rangle[0]) .
\end{equation}
The datum $\left(\mathcal{F}_B, \phi\right)$ is subject to the condition $\phi\left(\left.\mathcal{F}_B\right|_{\{\infty\}}\right)=B$.
It is proved by Yehao \cite{YH} that the generalized transversal slice also have a decomposition into symplectic leaves: $\overline{\mathcal{W}}_\mu^\lambda=\bigsqcup_{\mu \leq \nu \leq \lambda} \mathcal{W}_\mu^\nu$.

\end{define}
\begin{rmk}
We are interested in the case where $\lambda=0$ and $\mu=-\theta$, where $\theta$ be the quasi-minuscule coroot $\theta$ (i.e, the coroot corresponding to the highest root).  Then $\overline{\mathcal{W}}_\theta^0$  is isomorphic to the moduli of $B$ structures of the trivial $G$ bundle on $\mathbb{P}^1$ such that the $B$ structure at $\infty$ is the Borel subgroup $B \subset G$. In other words, it is the open Zastava:

\begin{equation}\label{4444}
\overline{\mathcal{W}}_{-\theta}^0 =\mathcal{W}_{-\theta}^0\cong \mathring{\mathcal{Z} }^{-w_0 \theta}=\mathring{\mathcal{Z}}^{\theta}
\end{equation}
\end{rmk}

\begin{define}
We define a morphism from the transversal slice to the Zastava space , i.e, 
$q^\lambda_\mu: \overline{\mathcal{W}}_\mu^\lambda \rightarrow \mathcal{Z}^{\alpha^*}$,  where $\alpha=\lambda-\mu$,    by the following:
\[ \left(\mathcal{F}_B, \phi\right)\mapsto  \left\{\mathcal{F}_B^{L_\chi}\left(\left\langle w_0 \lambda, \chi\right\rangle[0]\right) \subset V_\chi \otimes \mathcal{O}_{\mathbb{P}^1}\right\}\left.\right|_{\lambda \in \Lambda_G^{+}}\]
We need to check RHS indeed defines a point in Zastava space. Starting from an equivalence class of datum $\left(\mathcal{F}_B, \phi\right)$, we can associate a generalized $B$-structure of the trivial $G$-bundle. Since $\mathcal{F}_B^{L_\chi} \subset \mathcal{F}_B^{V_\chi}$ is a line subbundle, it follows from \eqref{4444} that $\mathcal{F}_B^{L_\chi}\left(\left\langle w_0 \lambda, \chi\right\rangle[0]\right)$ is a rank one locally free subsheaf of $V_\chi \otimes \mathcal{O}_{\mathbb{P}^1}$, and
$$
\operatorname{deg} \mathcal{F}_B^{L_\chi}\left(\left\langle w_0 \lambda, \chi\right\rangle[0]\right)=\left\langle w_0 \lambda-w_0 \mu, \chi\right\rangle
.$$
The map $\chi \mapsto \mathcal{F}_B^{L_\chi}\left(\left\langle w_0 \lambda, \chi\right\rangle[0]\right) \subset V_\chi \otimes \mathcal{O}_{\mathbb{P}^1}$  satisfies the Plücker relations,  moreover, we have $\left.\mathcal{F}_B^{L_\chi}\left(\left\langle w_0 \lambda, \chi\right\rangle[0]\right)\right|_{\infty}$ being the highest weight line of $V_\chi$, therefore the collection \[\left\{\mathcal{F}_B^{L_\chi}\left(\left\langle w_0 \lambda, \chi\right\rangle[0]\right) \subset V_\chi \otimes \mathcal{O}_{\mathbb{P}^1}\right\}\left.\right|_{\lambda \in \Lambda_G^{+}}\]defines a quasimap, i.e, a point in the Zastava.
\end{define}
\begin{rmk}
When $\lambda=0, q^0_{-\theta}$ is identified with the natural embedding $\mathring{\mathcal{Z}}^\theta=\mathring{\mathcal{Z}}^{-w_0 \theta} \subset \mathcal{Z}^{-w_0 \theta}=\mathcal{Z}^\theta$, using the isomorphism \eqref{4444}.
\end{rmk}

\begin{define}
Define the shift morphism $s_{\mu, \nu}^\lambda: \overline{\mathcal{W}}_\mu^\lambda \rightarrow \overline{\mathcal{W}}_{\mu+\nu}^{\lambda+\nu}$  by mapping a moduli data $\left(\mathcal{F}_B, \phi\right) \in \overline{\mathcal{W}}_\mu^\lambda$ to $\left(\mathcal{F}_B\left(-w_0 \nu[0]\right), \phi^{\prime}\right) \in$ $\overline{\mathcal{W}}_{\mu+\nu}^{\lambda+\nu}$, where $\mathcal{F}_B\left(-w_0 \nu[0]\right)$ is the twist of $\mathcal{F}_B$ by the colored divisor $-w_0 \nu[0]$, and $\phi^{\prime}$ is the modular definition of $q^\lambda_\mu: \overline{\mathcal{W}}_\mu^\lambda \rightarrow \mathcal{Z}^{\alpha^*}$, one can easily see that the shift morphisms commute with projections to Zastava space.
\end{define}

\begin{lma}

 The shift morphism $ s_{-\theta, \theta}^0: \overline{\mathcal{W}}_{-\theta}^0 \rightarrow \overline{\mathcal{W}}_{0}^\theta $ is an open embedding. 

\end{lma}
\begin{proof}
The map $q_{-\theta}^0$ identifies $\overline{\mathcal{W}}_{-\theta}^0$ with the open subset $\mathring{\mathcal{Z}}^{\theta} \subset \mathcal{Z}^{\theta}$, and $s_{-\theta, \theta}^0$ becomes a section of $q_{0}^\theta$ over $\mathring{\mathcal{Z} }^{\theta}$ upon this identification. Since both $\overline{\mathcal{W}}_{-\theta}^0$ and $\overline{\mathcal{W}}_{0}^\theta$ are integral scheme of the same dimension, $s_{-\theta, \theta}^0$ is an open embedding.
\end{proof}

\subsection{Integrable system on Zastava}

We recall the integrable system on Zastava space. For any $\alpha$ in the coweight lattice of $G$, $\alpha=\sum_{1\leq i\leq l}\alpha_i \varpi_i^\vee$ with $\alpha_i$ be integers, we define $\mathbb{A}^\alpha:=\prod_{1\leq i \leq l}\mathbb{A}^{\alpha_i}$. It is called the configuration space of colored divisors of multidegree $\alpha$. The integrable system on Zastava space is a morphism $\pi: \mathcal{Z}^\alpha \rightarrow \mathbb{A}^\alpha$, defined by mapping a quasimap $\left\{\mathcal{L}_\chi \subset\right.$ $\left.V_\chi \otimes \mathcal{O}_{\mathbb{P}^1}\right\}_{\chi \in \check{\Lambda}_G^{+}}$ to the colored divisor $\mathcal{L}_\chi \subset L_\chi \otimes \mathcal{O}_{\mathbb{P}^1}$, where $L_\chi$ is the highest weight line of the irreducible $G$-module $V_\chi$ of highest weight $\chi$, and the inclusion $\mathcal{L}_\chi \hookrightarrow L_\chi \otimes \mathcal{O}_{\mathbb{P}^1}$ is the composition of the projection $V_\chi \rightarrow L_\chi$ and the inclusion of subsheaf $\mathcal{L}_\chi \subset V_\chi \otimes \mathcal{O}_{\mathbb{P}^1}$.
It is known that the fibers of $\pi: \mathcal{Z}^\alpha \rightarrow \mathbb{A}^\alpha$ have dimension $\langle\alpha, \rho\rangle$, which is half of the dimension of $\mathcal{Z}^\alpha$, and the Poisson structure vanishes on the subalgebra $\pi^*\left(\mathbb{C}\left[\mathbb{A}^\alpha\right]\right) \subset \mathbb{C}\left[\mathcal{Z}^\alpha\right]$.
In particular $\pi: \mathcal{Z}^\alpha \rightarrow \mathbb{A}^\alpha$ is an integrable system. 
\begin{lma}\label{lem2.7}
There is an isomorphism $\overline{\mathcal{W}^\theta_0}\cong\overline{\mathcal{O}_{\text{min}}}$  given by extending the exponential map 
exp: $\mathcal{O}_{\text{min}}\rightarrow G(\mathbb{C}((t)))\rightarrow Gr_G(\mathbb{C})$ .
\end{lma}
\begin{proof}
Lusztig (\cite{lus81}, pp.175) first pointed out this isomorphism in type A in an equivalent way. For a general proof, see \cite[lemma 2.3.3]{mal05}. The proof uses a construction in  \cite[pp. 182-183]{bei91}, where one can see the isomorphism is defined by taking exponent on $\mathcal{O}_{\text{min}}$, but the parameter $t$ is chosen to be inverse so that the image of the exp map will fall in the transversal slice.  
\end{proof}

To summarize, we have the following commutative diagram, where the horizontal line are embeddings.

\begin{equation}\label{miao}
\centering
\begin{tikzcd}
\overline{\mathcal{W}^0_{-\theta}}\cong\mathring{\mathcal{Z}}^\theta\ar[r,hook," s_{-\theta, \theta}^0"  ]\ar[drr]& \overline{\mathcal{W}^\theta_0} \overset{\exp}{\stackrel{\cong}{\longleftarrow}}\overline{\mathcal{O}_{\text{min}}}\ar[dr]\ar[r,hook,"q^\theta_0"]&\mathcal{Z}^\theta\ar[d,"\pi"]\\&  &\mathbb{A}^\theta
\end{tikzcd}
\end{equation}

We will compute the integrable system on $\overline{\mathcal{O}}_{\text{min}}$ induced form $\pi$, using the following lemma:
\begin{lma}\label{2.9}
The integrable system $\pi$ is given by the equations of Schubert divisors $\{X_j\}_j$ of the flag variety $G/B^-$, where the fundamental weights are labelled by $j$,  $X_j=B^+s_j B^-$, $s_j$ is the simple reflection. 
\end{lma}
\begin{proof}
To determine $\pi$, it is enough to associate the system $\{\mathcal{L}_{\varpi_i} \subset V_{\varpi_i} \otimes \mathcal{O}_{\mathbb{P}^1}\}_{\varpi_i \in \check{\Lambda}_G^{+}}$ to the colored divisor  $\mathcal{L}_{\varpi_i}\subset L_{\varpi_i}\otimes \mathcal{O}_{\mathbb{P^1}}$ . 
For each fundamental representation $V_{\varpi_i}$ of $G$,  we can associate it a trivial bundle $\mathcal{V}_i$.
Each point $[gB^-]\in G/B^-$ gives a highest weight vector $L_{\varpi_i}[gB^-]$ of $V_{\varpi_i}$ . Denote the highest weight vector in $V_{\varpi_i}$ with respect to $B^-$ by $L_{\varpi_i}$, then the map $\{L_{\varpi_i}[gB^-]\hookrightarrow V_{\varpi_i} \rightarrow L_{\varpi_i}\}_{\{gB^-\in \mathbb{P}^1\hookrightarrow G/B^-\} }$ gives an ideal $I_i$  of the structure sheaf of vanishing locus $\mathcal{L}_{\varpi_i}\hookrightarrow L_{\varpi_i}\otimes \mathcal{O}_{\mathbb{P^1}}$.  This map vanishes on the divisor $B^+s_iB^-$, and is isomorphism on the divisor $B^+s_jB^-$ for $j\neq i$. Therefore the divisor only vanish on the $i$ -th Schubert divisor. A comparison of degree in the following shows that this ideal is the Schubert divisor itself. \\
The degree of $I_i$ is computed by its class as the line bundle $\mathcal{L}_{\varpi_i}$ in the Picard group $\mathrm{Pic}\,G/B^-$, which is isomorphic to the $G$-equivariant Picard group 
$\mathrm{Pic}_G\,G/B^-$. However, we have $\mathrm{Pic}_G\,G/B^-\cong \mathrm{Pic}\,[\mathrm{pt}/B^-]\cong \mathrm{Hom}(B^-,\mathbb{G}_m)$, where   $[\mathrm{pt}/B^-]$ is the classifying stack of $B$, which is isomorphic to the character group. 
The image of $\mathcal{L}_{\varpi_i}$ under these isomorphisms are given by noticing that this line bundle is $B^-$ invariant, its fiber at point $[gB^-]$ is a dimension-one $B^-$-representation, which is a character of $B^-$, we regard this character as the degree of $\mathcal{L}_{\varpi_i}$, denoted by $\chi_i$. All $\{\chi_i\}_i$ form a basis of the Picard group. \\On the other hand, there is a homomorphism from divisors on $G/B^-$ to its Picard group. Denote the subgroup spanned by Shubert divisors in $\mathrm{Div}(G/B^-)$ by $S_{G/B^-}$.  
Since the complement of the union of Schubert divisors $G/B^--\cup_i X_i$  is an affine sapce,  the restriction homomorphism $S_{G/B^-}\to\mathrm{Pic}\, G/B^-$ is surjective. Observe that $\mathrm{rank} (S_{G/B^-})=\mathrm{rank} (\mathrm{Pic}\,G/B^-)$, therefore the image of Schubert divisors $X_i$ forms another basis in the Picard group. \\
Since each $\chi_i$ is supported on $X_i$, there exists a positive integer $n$ such that $\chi_i=n [X_i]$  in the image of $\mathrm{Div}\,(G/B^-)$ in $\mathrm{Pic}\,(G/B^-)$. Notice that both $\{\chi\}_i$ and $[X_i]$ are basis of $\mathrm{Pic}\,(G/B^-)$, we have $n=1$.

\end{proof}

\subsection{Compare Poisson structure on generalized slice and minimal orbit}

\begin{prop}\label{2.10}
The exponential map $\exp:\,\overline{\mathcal{O}_{min}}\to \mathrm{Gr}_G(\C)$ is Poisson, with the Poisson structure on the left side be the Kirillov-Kostant strucuture, and on the right side is the Poisson structure induced by the Manin triple $\mathfrak{g}((t^{-1})),t^{-1}\mathfrak{g}[[t^{-1}]],\mathfrak{g}[t])$, with the bilinear form on $\mathfrak{g}((t^{-1}))$ defined by 
\[(f(t),g(t)):=-\mathrm{Res}_{t\to 0}\langle f(t),g(t)\rangle\]
where $\langle,\rangle$ is the Killing form on $\mathfrak{g}$.
\end{prop}
\begin{proof}

Let $\theta$ be the coweight of $\mathfrak{g}=\mathrm{Lie}\,G$ corresponding to the maximal root $\theta^\vee$ of $\mathfrak{g}$.  The image of this exponential map is the closure of the orbit of affine grassmannian associated to quasi-minuscule coweight $\theta$, we denote it as  $\overline{\mathrm{Gr}^\theta}$, and $\mathrm{Gr}^\theta=G[[t]]\cdot t^\theta =G/P_\theta^- $ is a projective variety, where $P_\theta^-=\{g\in G| t^{-\theta}g t^\theta\in G[[t]]\}$ is the parabolic subgroup of G such that $\mathrm{Lie}\,P_\theta^-$ is the sum of Cartan and the root spaces corresponding to roots $\alpha$ with $(\alpha,\theta)\leq 0$.  

We know $\mathrm{Gr}^\theta$ is a Poisson subvariety of $\mathrm{Gr}_G(\C)$, and from \cite[pp.183]{bei91} we know that the exponential map is $G$-equivariant, so it will induce a $G$-invariant map on tangent spaces.
we consider the tangent space at the nagative maximal root vector $x=E_{-\theta^\vee}$ in the nilpotent orbit. The group $G$ acts on the tangent space $T_x$ of the minimal nilpotent orbit at $x$, and its stablizer $K$ is the kernel of the character $\theta^\vee:\,P_\theta\to\C^\times$. In particular,
the unipotent $U^-$ (radical of Borel) is in the stabilizer $K$ and also in the stabilizer $P_\theta ^-$. Consider $T_x$ as a $U^-$-module, it is generated by the highest root vector $E_{\theta^\vee}$, and $E_{\theta^\vee}$ is the unique such generator up to scaling.  Since we know there are $K$-invariant symplectic structure on $T_x$ (namely the KKS structure), then as a $U^-$-module, $T_x\cong T_x^*$, so $\mathrm{Isom}_{U^-}(T_x,T_x^*)=\C^\times$, in particular any $K$-invariant symplectic form must be proportional to this one. 
Consider the tangent space at the origin of $Gr^\theta$, we showed that the symplectic structure coming from the minimal nilpotent orbit and the symplectic structure from the Manin triple differs by a scalar multiplication. Since the LHS and RHS of the exponential map are both $G$-homogeneous space,  the symplectic structure at tangent space implies that it is a Poisson map.

\end{proof}

\section{Computation of Hamiltonians for every type}\label{s4}

\subsection{Chevalley basis}

In this subsection, we will compute the Hamiltonians on minimal orbit for complex simple Lie algebra for every type  in terms of Chevalley basis. 

From lemma \ref{2.9}, the Hamiltonian functions of our integrable system on minimal nilpotent orbit $\mathcal{O}_{min}$ are given by the equations of Schubert divisors,   For opposite Borel $B^{-}$ in $G$, denote Schubert divisors of $G/B^-$ by $X_i:=B^{+} s_i B^{-}$,
$v_i$ be a vector of the $B^{-}$-stable line in $V_{\varpi_i}$, $v^i$ be a vector in the $B^{+}$-stable line of the dual representation $\left(V_{\varpi_i}\right)^*$, and $\langle,\rangle$ is the natural pairing. 
It was shown in \cite{bern82} that the equations of Schubert divisors $X_i$ are the zeroes of functions $g \mapsto\left\langle  v^i, g v_i\right\rangle$ on $G$ (see also \cite[pp.237]{cg}). Let $x \in \mathcal{O}_{\min }=\mathrm{Ad}_G \cdot e_{\theta^\vee}$, by lemma \ref{2.9} and proposition \ref{2.10}, the Hamiltonians of the integrable system on $\mathcal{O}_{\min }$ are the coefficients of $t^{-n}\,(n\geq1)$ for the functions $x\mapsto\langle v^i,\exp (t^{-1} x)v_i\rangle$, where $1\leq i\leq l=\mathrm{rank}(\mathfrak{g})$, and we call the coefficients of $t^{-n}$ be \emph{Hamiltonians of order-$n$}.  More explicitly, we shall derive the coefficient of $t^{-n}$ of the equations $x\mapsto \langle v^i,\mathrm{exp}\,(t^{-1}x)v_i\rangle$ , where $x\in\mathfrak{g}$ is expanded in a set of Chevalley basis. The order-$1$ Hamiltonians are called linear, and order-$2$ are called quadratic. Notice when order higher than $1$ they may vanish when restricted to minimal nilpotent orbit. In next subsection we will use another coordinate to calculate and show when do they vanish. 
 \\ \hspace*{\fill}\\
 
\begin{prop}\label{xiaobaobao}
\begin{enumerate}[label=(\arabic*)]
\item The linear Hamiltonians are the linear form $\varpi_i\in \mathbb{C}[\mathfrak{g}]$.
 \item  \label{2222}The quadratic Hamiltonians are $\varpi_i^2+\sum_{\alpha\in R^+-R_{P_i}}\langle \varpi_i,\alpha^\vee\rangle e^*_\alpha f^*_\alpha\in\mathbb{C}[\mathfrak{g}]$.
 \item The Hamiltonians of any order associated to $\varpi_i$  lie in $\mathbb{C}[\mathfrak{g}]^{L_i}\,\mathrm{mod}\langle \mathbb{C}[\mathfrak{g}]^G_+\rangle$ .
 
 \item When the fundamental weight $\varpi_i$ is cominuscule, then the quadratic Hamiltonians are the dual Casimir element of $\mathfrak{g}$ minus the dual Casimir element of $L_i$. \end{enumerate}
\end{prop}
\begin{proof}
\begin{enumerate}[label=(\arabic*)]
\item  Write elements in $\mathfrak{g}$ in the basis of the chosen Chevalley generators.
Each "$f$" in the triple draws $v_i$ to a lower weight space, which will not contribute when pairing with $v_i^\vee$, and $v_i$ is vanished under the action of "$e$". \\
To calculate the coefficient of $t^{-1}$ in $\langle \mathrm{exp}  (t^{-1}A) v_i,v^i\rangle$  , the only possible nonzero coefficient is contributed by these "$h_\alpha$", where $\alpha$ is a positive root. Note that
\[  h_\alpha v_{i}=\langle \varpi_i,\alpha^\vee\rangle v_i ,  \]

that is, the, we get the linear form $\varpi_i\in\mathbb{C}[\mathfrak{g}]$.

\item When we calculate $\langle(\mathrm{exp}(t^{-1}A) v_i,v^i\rangle$,  the second-order term can arise in two ways: first is to apply Cartan element $h_\alpha$ , i.e, the action of $h_\alpha^2$, the second is by first go to the lower weight space by $f_\alpha$ and then go up by $e_\alpha$.  
\[h_\alpha^2 v_i=\langle \varpi_i,  \alpha^\vee\rangle^2,\quad e_\alpha f_\alpha v_i=[e_\alpha, f_\alpha]v_i=\langle \varpi_i,\alpha^\vee\rangle v_i \]
Therefore we get the second Hamiltonian is $\varpi_i^2+\sum\langle \varpi_i,\alpha^\vee\rangle e_\alpha^* f_\alpha^*\in\mathbb{C}[\mathfrak{g}]$
where $\alpha$ runs over positive root $R^+$. Notice that for the maximal parabolic subgroup $P_i$ associated to fundamental weight $\varpi_i$, the associated root $R_{P_i}$ are linear combinations of $\alpha_j$ ($j\neq i)$, and for any $\alpha\in R_{P_i}$, $\langle \varpi_i, \alpha\rangle=0$. Hence we can sum only in $R^+-R_{P_i}$. 

\item
The fact that Hamiltonians associated with $\varpi_i$ lie in $\mathbb{C}[\mathfrak{g}]^{L_i}$is proved  by noticing the subrepresentation of the Levi factor $L_i$ on the fundamental representation $V_{\varpi_i}$ fixes the highest weight, i.e, $L_iv_i=v_i$.  Indeed, by Bruhat decomposition, it suffices to show the Borel $B_{L_i}$ and the Weyl group $W_{P_i}$ all fix $v_{i}$, which follows respectively by noticing that $B_{L_i}\subset B$ and $Stab_\Lambda(\varpi_i)=W_{P_i}$.\\
Moreover, recall for any $f\in \mathbb{C}[\mathfrak{g}]^G_+$, $f$ acts on any nilpotent element is zero. Hence we can modulo by $\langle\mathbb{C}[\mathfrak{g}]^G_+\rangle$. 

\item When $\varpi_i$ is cominuscule,  i.e,  the corresponding coweight $\varpi_i^\vee$ is minuscule,   $\langle \varpi_i,\alpha^\vee\rangle=1$ for $\alpha\in R^+-R_{P_i}$, which  gives the desired Casimir form.  In particular when $G$ is of type $A$, all quadratic Hamiltonians can be written as the substraction of two Casimirs. 

\end{enumerate}

\end{proof}

Next we give a formula for  Hamiltonians of any order.  Let $V$ be any fundamental representation with weight $\varpi$ of the complex simple Lie algebra $\mathfrak{g}$, let $v_m$, $v^m$ be respectively the highest weight vector and lowest weight vector. Pick a basis $(e_i)_{i=1,\cdots, l}$ of $\mathfrak{g}$,  recall we defined the Casimir element to be
    \[\mathrm{Cas}:=\sum_{i=1}^{l}e_i\otimes e_i^*\in U(g)\otimes U(g)^*,\]
    and we can write its power using Sweedler's notation:
    \[\mathrm{Cas}^n=\mathrm{Cas}^n_{(1)}\otimes\mathrm{Cas}_{(2)}^n\]

    \begin{prop}\label{145}
    For any $x\in\mathfrak{g}$, the order-$n$ Hamiltonians associated with fundamental weight $\varpi$ are the matrix coefficients 
    \begin{equation} \label{33}\langle v^m,x^nv_m\rangle=\mathrm{Cas}_{(2)}^n\langle\mathrm{Cas}_{(1)}^n v_m, v^m\rangle (x).\end{equation}
    
\end{prop}
\begin{proof}

By the discussion at the beginning of \ref{345}, we see the order-$n$ Hamiltonians are exactly those matrix coefficients $\langle v^m,x^nv_m\rangle$ after expanding $\langle v^m,\exp (t^{-1} x)v_m\rangle $ into powers of $t$.   For any element in the Lie algebra, expand the matrix coefficients in the chosen basis, i.e, $x=a_1e_1+\cdots a_le_l\in\mathfrak{g}$,  then

\begin{align*}
    \langle v^m,x^nv_m\rangle&=\langle v^m, (a_1e_1+\cdots+a_le_l)\cdots(a_1e_1+\cdots a_le_l)v_m\rangle
    \\&=\sum_{i_1,\cdots,i_n}a_{i_n}a_{i_{n-1}}\cdots a_{i_1}\langle e_{i_n}e_{i_{n-1}}\cdots e_{i_1}v_m,v^m\rangle \\&=\sum_{i_1,\cdots,i_n}e_{i_n}^*\cdots e_{i_1}^*\langle e_{i_n}\cdots e_{i_1}v_m,v^m\rangle x ,
\end{align*}

    where in the above equation we used $a_i=e_i^* x$, and $a_{i_n}\cdots a_{i_1}=(e_{i_n}^*\cdots e_{i_1}^*)x$. On the other hand,
    
\begin{align*}\mathrm{Cas}^n&=(e_1\otimes e^*_1+\cdots +e_l\otimes e^*_l)\cdots (e_1\otimes e^*_1+\cdots +e_l\otimes e^*_l)\\&=\sum_{i_1,\cdots,i_n}e_{i_n}\cdots e_{i_1}\otimes e^*_{i_n}\cdots e^*_{i_1} \end{align*}

Hence $\mathrm{Cas}^n_{(1)}$ are $e_{i_n}\cdots e_{i_1}$, $\mathrm{Cas}^n_{(2)}$ are $e_{i_n}^*\cdots e_{i_1}^*$. Substituting to $\mathrm{Cas}_{(2)}^n\langle\mathrm{Cas}_{(1)}^n v_m, v^m\rangle (x)$, we see immediately the equality holds.
\end{proof}
\begin{rmk}
    The equation \eqref{33} holds for any irreducible representation of $\mathfrak{g}$.
\end{rmk}

Next we show another formulation of Hamiltonians when $\mathfrak{g}$ is classical Lie algebra. We first develop a lemma in linear algebra. Let $A$  be inside any of the classical Lie groups: $\mathrm{SL}_{n+1}(\C)$, $\mathrm{Sp}_{2n}(\mathbb{C})$, $\mathrm{SO}_{2n}(\C)$,  $\mathrm{SO}_{2n+1}(\mathbb{C})$.   If $A$ lies in $\mathrm{SL}_{n+1}$ or $\mathrm{Sp}_{2n}$ , pick $k\in\{1,\cdots, n\}$. If $A\in\mathrm{SO}_{2n}$ , pick $k\in\{1, \cdots, n-2\}$. If $A\in SO_{2n+1}$, pick $ k\in\{1, \cdots, n-1\}$. 
\begin{lma}\label{xiaoai}
For $r\in\N_+$, the degree-$r$ matrix coefficients associated to $k$-th fundamental weight $\varpi_k$  are
\[\langle v^{\omega_k},A^r v_{\omega_k}\rangle=\det((A^r)_{k\times k})\], where $(,)_{k\times k}$ signifies the upper-left $k\times k$ block of that matrix.  
\end{lma}

\begin{proof}
    For $A\in\mathrm{GL}_{N}(\C)$, set $A^r=(a_{ij})$, the highest weght vector of $\varpi_k$ is $v_m=e_1\wedge\cdots \wedge e_k$, where $e_i$ is the $i$-th standard basis vector of $\mathbb{C}^N$. Denote $u_j=\sum_{i=1}^N a_{ij}e_i$, then 
\[\langle v^m,A^rv_m\rangle=\langle v^m,(A^re_1)\wedge(A^re_2)\cdots\wedge (A^re_k)\rangle=\langle v^m, u_1\wedge\cdots\wedge u_k\rangle. \]
By proposition 3.3 in \cite[pp.84]{Yokonuma}, $u_1\wedge\cdots \wedge u_k=\sum_I (A^r)_Ie_I$, where $I$ runs through all strictly increasing sequence of integers $I=(1\leq i_1<\cdots<i_k\leq N)$, $e_I=e_{i_1}\wedge\cdots\wedge e_{i_k}\in\bigwedge^k(\C^n)$, and $(A^r)_I$ is the determinant of $k\times k$ matrix obtained by extracting lines $i_1,\cdots,i_k$ from $
N\times k$ matrix $(A^r)_{N\times k}=(a_{ij})_{1\leq i\leq N, 1\leq j\leq k}$. 
After pair with $v^m$, only first $k\times k$-part left, i.e, we get  $\det((A^r)_{k\times k})$.

 \end{proof}

 In the following, let the integer $k$  be the same as in the lemma \ref{xiaoai}.
\begin{prop}\label{999}
    For any matrix $x\in\mathfrak{sl}_n(\C)$,  $\mathfrak{sp}_{2n}$, $\mathfrak{so}_{2n}(\C)$, $\mathfrak{so}_{2n+1}(\C)$, the Hamiltonians associated to $\varpi_k$  is \[\langle v^{\omega_k},x^m v_{\omega_k}\rangle=\text{coefficient of }\, t^m \,\text{in}\, \det[(1+tx+\frac{t^2x^2}{2}+\cdots)_{k\times k}]\]
\end{prop}
\begin{proof}
    Write $A=e^x\in\mathrm{SL}_n(\mathbb{C})$, then $A^n=e^{nx}$. From lemma \ref{xiaoai}, 
\[\langle v^{\omega_k}, (1+nx+\frac{n^2x^2}{2}+\cdots)v_{\omega_k}\rangle=\det[(1+nx+\frac{n^2x^2}{2}+\cdots)_{k\times k}].\]
Writing $x=tx'$ and substituting to the above formular, we see the proposition follows.
\end{proof}
\begin{rmk}\label{1000}
    If we are considering $x$ in the minimal nilpotent orbit, then $x^2=0$ . For type $A$ and type $C$,  $\mathrm{rank}(x)=1$,  so we get the coefficient of $t$ in $\det[(1+tx)_{k\times k}]$, which is $\mathrm{tr}(x_{k\times k})$. For type $B$ and type $D$ when the fundamental representation is not the spinor representation, $\mathrm{rank}(x)=2$, we get the coefficient of $t$ in $\det[(1+tx)_{k\times k}]$ and the coefficient of $t^2$ in  $\det[(1+tx)_{k\times k}]$  , the latter is $\mathrm{Tr}(\bigwedge^2 B_{k\times k})$, which is a coefficient of the characteristic polynomial of $x_{k\times k}$. 
For type $B$ and $D$ when the fundamental representation is the spinor representation, we need to use the vanishing results by Heisenberg algebra basis, as in theorem \ref{3.1}.
\end{rmk}

\subsection{ Heisenberg algebra's basis}
We have
\[\mathcal{O}_{\text{min}}=Gf_{\theta^\vee}=\bigcup_{w\in W}BwB f_{\theta^\vee}=\bigcup_{w\in W}Bf_{w\theta^\vee}=\bigcup_{\alpha\in\Phi_{long}}B\cdot f_{\alpha}\]

In particular $B f_{\theta^\vee}$  is a dense open set in $\mathcal{O}_{min}$, in the following we show it is the same as $\C^*$ times the orbit of Heisenberg group $H$ of Kostant's Heisenberg algebra. We let $n^\star:=\mathrm{span}\{\varphi\in R^+|\, \langle \varphi, \theta\rangle \neq 0\}$ be Kostant's Heisenberg algebra. In \cite[pp.252]{kostant85} it is shown that $n^\star$ is a Heisenberg algebra and has dimension $2h^\vee-3$.
We write a basis of the Lie algebra of $B$ as $\{H_1,\cdots, H_l, Y_1,\cdots Y_m, X_1, \cdots, X_{2h^\vee-3}\}$, where $H_1,\cdots, H_l$ are Cartan parts, $Y_1,\cdots, Y_m$ are those positive roots which are perpendicular to the highest root $\theta^\vee$, and $X_1,\cdots,X_{2h^\vee-3}$ are Kostant's roots.

\begin{lma}\label{zisha}
    There is an isomorphism of varieties between $B$-orbit of $f_{\theta^\vee}$ and $C^*$ times the Heisenberg orbit of $f_{\theta^\vee}$, i.e,
    \[Bf_{\theta^\vee}\cong \C^*\times Hf_{\theta^\vee}\cong \C^*\times \mathfrak{n}^\star \]
\end{lma}
\begin{proof}
The Borel $B$ has a Levi decomposition as a semidirect product $B=TN$, with $T$ the maximal torus and $N$ its unipotent radical. From \cite[proposition. 8.2.1]{sp} , there is an isomorphism of varieties 
 $\phi:\,\mathbb{C}^{m+2h^\vee-3}\xrightarrow{\sim} N$ with $\phi(x_1,\cdots,x_{2h^\vee-3},y_1,\cdots,y_m)=\exp(x_1X_1)\cdots\exp(x_{2h^{\vee}-3}X_{2h^\vee-3})\exp(y_1Y_1)\cdots\exp(y_mY_m).$
Hence $Bf_{\check{\theta}}=NTf_{\check{\theta}}=\mathbb{C}^*\times Nf_{\check{\theta}}=
\mathbb{C}^*\times \exp(x_1X_1)\cdots\exp(x_{2h^{\vee}-3}X_{2h^\vee-3})\exp(y_1Y_1)\cdots\exp(y_mY_m)
f_{\check{\theta}}.
$
Notice the action of $\exp(y_1Y_1)\cdots\exp(y_mY_m)$ stabilizes $f_{\theta^\vee}$, which following from the fact that any positive root vector $e_\gamma$ annihilates $f_{\theta^\vee}$ if and only if $(\theta^\vee, \gamma)=0$, see \cite[8.4,9.4]{humphreys12}  , i.e, $\gamma\in\{Y_1,\cdots, Y_m\}$. Therefore, we have the following surjection map 

\begin{equation}\label{mmn}\C^*\times \mathfrak{n}^\star \xrightarrow{\exp} \C^*\times H\to \C^*\times Hf_{\theta^\vee}\cong B f_{\theta^\vee}.     \end{equation}

From $\dim\,\C^*\times\mathfrak{n}^\star=2h^\vee-2$, $\dim B f_{\theta^\vee}=\dim \,\mathcal{O}_{min}=2h^\vee-2$,  we obtain the map is an isomorphism.
\end{proof}
In the following we identify $\mathbb{C}^*\times \mathfrak{n}^\star$ as a subset in $\mathcal{O}_{\min}$ thanks to the previous lemma \ref{zisha}. The part of $\mathbb{C}^*$ gives a scalar multiplication factor on the Hamiltonian function, so it will be sufficient for us to calculate Hamiltonians on $\mathfrak{n}^\star$. 
Let $x_0,x_1,\cdots$ be the coefficients of positive roots in $n^\star$.  For $x=x_1 \varphi_{1}+x_2\varphi_{2}+\cdots+x_{2h^\vee-3}\varphi_{2h^\vee-3}\in \mathfrak{n}^\star$,  denote by $E_i$ the element in $U(n^\star)$ corresponding to the root $\varphi_i$, define 
$\mathrm{Kos}:=\sum_{i=1}^{2h^\vee-3}E_i\otimes (E_i)^* \in U(\mathfrak{n}^\star)\otimes U(\mathfrak{n}^\star)^*$.
\begin{prop}
The degree-$r$ Hamiltonian corresponding to the fundamental weight $\varpi_i$ is 
\begin{equation}\label{778}\langle v^i, \exp(x)f_\theta^r v_i\rangle=\sum_{k=1}^{l}\frac{1}{k!}\mathrm({Kos}^k_{(2)})\langle \mathrm{Kos}^k_{(1)}f_\theta^r v_i,v^i\rangle(x)\end{equation}
where as before we use Sweedle notation $\mathrm{Kos}^n=\mathrm{Kos}_{(1)}^n\otimes\mathrm{Kos}_{(2)}^n$.
\end{prop}
\begin{proof}
    The $H$- orbit of $f_{\theta}$  is $Hf_{\theta}=\{\exp(x)f_{\theta}\exp(-x):\,x\in \mathfrak{n}^\star\}$,  let $v_m$ be a highest weight vector of fundamental representation, $y=\exp(x)f_{\theta}\exp(-x)\in Hf_\theta$, the matrix coefficient  
 
\begin{align*}\langle v^m, \exp(y)v_m\rangle &=\langle v^m,\exp(x)\exp(tf_{\theta})\exp(-x)v_m\rangle=\langle v^m, \exp(x)\exp(tf_{\theta})v_m\rangle\\&=\langle v^m, \exp(x)(v_m+tf_\theta v_m+t^2\frac{f_\theta^2}{2}v_m+\cdots)\rangle\end{align*}\\
Extracting powers of $t$, same calculation as in \ref{145} yields the results.
\end{proof}

\begin{eg}In this example we write explicitly the Hamiltonians from $\mathfrak{so}(5)$ to $\mathfrak{so}(10)$ and $\mathfrak{g}_2$ using Heisenberg basis. First we label the Heisenberg's root as the follwing table :
\\ \hspace*{\fill}\\
\resizebox{\linewidth}{!}{
$\begin{array}{ |ccccccc|  }
\hline
\text{Roots in Heisenberg}&B_2&D_3&B_3&D_4&B_4&D_5 \\
\hline
E[0]&\alpha_2&\alpha_2 &\alpha_2 &\alpha_2&\alpha_2 &\alpha_2 \\
\hline
E[1]&\alpha_1+2\alpha_2&\alpha_3&\alpha_1+\alpha_2&\alpha_1+\alpha_2&\alpha_1+\alpha_2&\alpha_1+\alpha_2
\\
\hline
E[2]&\alpha_1+\alpha_2&\alpha_1+\alpha_2&\alpha_2+2\alpha_3&\alpha_2+\alpha_3&\alpha_2+\alpha_3&\alpha_2+\alpha_3
\\
\hline
E[3]&&\alpha_1+\alpha_3&\alpha_2+\alpha_3&\alpha_2+\alpha_4&\alpha_1+\alpha_2+\alpha_3&\alpha_1+\alpha_2+\alpha_3    \\
\hline
E[4]&&\alpha_1+\alpha_2+\alpha_3&\alpha_1+\alpha_2+2\alpha_3&\alpha_1+\alpha_2+\alpha_3&\alpha_2+\alpha_3+2\alpha_4&\alpha_2+\alpha_3+\alpha_4 \\
\hline
E[5]&&&\alpha_1+\alpha_2+\alpha_3&\alpha_1+\alpha_2+\alpha_4&\alpha_2+\alpha_3+\alpha_4&\alpha_2+\alpha_3+\alpha_5\\
\hline
E[6]&&&\alpha_1+2\alpha_2+2\alpha_3&\alpha_2+\alpha_3+\alpha_4&\alpha_1+\alpha_2+\alpha_3+2\alpha_4&\alpha_1+\alpha_2+\alpha_3+\alpha_4\\
\hline
E[7]&&&&\alpha_1+\alpha_2+\alpha_3+\alpha_4&\alpha_2+2\alpha_3+2\alpha_4&\alpha_1+\alpha_2+\alpha_3+\alpha_5\\
\hline
E[8]&&&&\alpha_1+2\alpha_2+\alpha_3+\alpha_4&\alpha_1+\alpha_2+\alpha_3+\alpha_4&\alpha_2+\alpha_3+\alpha_4+\alpha_5\\
\hline
E[9]&&&&&\alpha_1+\alpha_2+2\alpha_3+2\alpha_4&\alpha_1+\alpha_2+\alpha_3+\alpha_4+\alpha_5\\
\hline
E[10]&&&&&\alpha_1+2\alpha_2+2\alpha_3+2\alpha_4&\alpha_2+2\alpha_3+\alpha_4+\alpha_5
\\
\hline
E[11]&&&&&&\alpha_1+\alpha_2+2\alpha_3+\alpha_4+\alpha_5\\
\hline
E[12]&&&&&&\alpha_1+2\alpha_2+2\alpha_3+\alpha_4+\alpha_5\\
\hline
\end{array}$}
\\ \hspace*{\fill}\\
Then
\begin{itemize}
    \item $\mathfrak{so}(5)$: 
    For a generic element $x=x_0E[0]+x_1E[1]+x_2E[2]$ inside the Heisenberg, we summerize the Hamiltonians in the table:\\ \hspace*{\fill}\\ 
    
    \[\begin{array}{|ccc|}
    \hline
    &\varpi_1&\varpi_2\\
    \hline
    \text{degree 1 parts} &x_0x_2+x_1&x_1\\
    \hline
    \text{degree 2 parts} &0&0\\
    \hline
     \end{array}
     \]
\\
    \item 
    $\mathfrak{so}(6)$: For a generic element $x=x_0E[0]+x_1E[1]+x_2E[2]+x_3E[3]+x_4E[4]$ inside the Heisenberg, we summerize the Hamiltonians in the table:    \\ \hspace*{\fill}\\
\[
    \begin{array}{|cccc|}
    \hline
    &\varpi_1&\varpi_2&\varpi_3\\
    \hline
    \text{Order-1} &\frac{1}{2}x_1x_2+\frac{1}{2}x_0x_3+x_4&\frac{1}{2}x_1x_2-\frac{1}{2}x_0x_3+x_4&-\frac{1}{2}x_1x_2+\frac{1}{2}x_0x_3+x_4    \\
    \hline
    \text{Order-2} &0&0&0\\
    \hline
     \end{array}
     \]
    \\
\item 
    $\mathfrak{so}(7)$: For a generic element $x=x_0E[0]+x_1E[1]+x_2E[2]+x_3E[3]+x_4E[4]+x_5E[5]+x_6E[6]$ inside the Heisenberg, we summerize the Hamiltonians in the table:      \\ \hspace*{\fill}\\    \resizebox{\linewidth}{!}{
    $\begin{array}{|cccc|}
    \hline
    &\varpi_1&\varpi_2&\varpi_3\\
    \hline
    \text{Order-1} &-\frac{1}{2}x_1x_2+\frac{1}{2}x_0x_4+x_3x_5+x_6&2x_6&\frac{1}{2}x_1x_2+\frac{1}{2}x_0x_4+x_6   \\
    \hline
    \text{Order-2} &0&-\frac{1}{4}x_1^2x_2^2 - \frac{1}{2}x_0x_1x_2x_4 - x_1x_3^2x_4 - \frac{1}{4}x_0^2x_4^2 &0\\& &+ x_1x_2x_3x_5 - x_0x_3x_4x_5 + x_0x_2x_5^2 + x_6^2&\\
    \hline
     \end{array}
     $
}
    \\
\item 
    $\mathfrak{so}(8)$: For a generic element $x=x_0E[0]+x_1E[1]+x_2E[2]+x_3E[3]+x_4E[4]+x_5E[5]+x_6E[6]+x_7E[7]+x_8E[8]$ inside the Heisenberg, we summerize the Hamiltonians in the table:         \\ \hspace*{\fill}\\
    \resizebox{\linewidth}{!}{   
    $\begin{array}{|ccccc|}
    \hline
    &\varpi_1&\varpi_2&\varpi_3&\varpi_4\\
    \hline
    \text{Order-1} &\frac{1}{2}x_3x_4+\frac{1}{2}x_2x_5-\frac{1}{2}x_1x_6   &2x_8 & \frac{1}{2}x_3x_4 - \frac{1}{2}x_2x_5 + \frac{1}{2}x_1x_6 &
    -\frac{1}{2}x_3x_4 + \frac{1}{2}x_2x_5 + \frac{1}{2}x_1x_6    \\
&+\frac{1}{2}x_0x_7+x_8 & &+ \frac{1}{2}x_0x_7 + x8&    + \frac{1}{2}x_0x_7 + x_8\\    \hline
    \text{Order-2} &0&-\frac{1}{4}x_3^2x_4^2+\frac{1}{2}x_2x_3x_4x_5-\frac{1}{4}x_2^2x_5^2&0 &0\\ & & +\frac{1}{2}x_1x_3x_4x_6+\frac{1}{2}x_1x_2x_5x_6+x_0x_4x_5x_6& &\\& &-\frac{1}{4}x_1^2x_6^2-x_1x_2x_3x_7-\frac{1}{2}x_0x_3x_4x_7& &\\& &-\frac{1}{2}x_0x_2x_5x_7-\frac{1}{2}x_0x_1x_6x_7-\frac{1}{4}x_0^2x_7^2+x_8^2&&\\    \hline
     \end{array}
    $}
\\ \hspace*{\fill}\\
\item 
    $\mathfrak{so}(9)$: For a generic element $x=x_0E[0]+x_1E[1]+x_2E[2]+x_3E[3]+x_4E[4]+x_5E[5]+x_6E[6]+x_7E[7]+x_8E[8]+x_9E[9]+x_{10}E[10]$ inside the Heisenberg, we summerize the Hamiltonian in the table:     \\ \hspace*{\fill}\\
    \resizebox{\linewidth}{!}{   $ \begin{array}{|ccccc|}
    \hline
    &\varpi_1&\varpi_2&\varpi_3&\varpi_4\\
    \hline
    \text{Order-1} &\frac{1}{2}x_3x_4+\frac{1}{2}x_2x_6-\frac{1}{2}x_1x_7   &2x_{10} & x_1x_7+x_0x_9+2x_{10} &
    -\frac{1}{2}x_3x_4 + \frac{1}{2}x_2x_6      \\ &-x_5x_8+\frac{1}{2}x_0x_9+x_10& & &+ \frac{1}{2}x_1x_7+\frac{1}{2}x_0x_9+x_{10} \\
    \hline
    \text{Order-2} &0&
    -\frac{1}{4}x_3^2x_4^2+\frac{1}{2}x_2x_3x_4x_6
    &
    -\frac{1}{4}x_3^2x_4^2+\frac{1}{2}x_2x_3x_4x_6
    &0\\    
    & & -x_3x_5^2x_6-\frac{1}{4}x_2^2x_6^2 &    -x_3x_5^2x_6-\frac{1}{4}x_2^2x_6^2& \\
    & & +\frac{1}{4}x_1^2x_7^2+x_3x_4x_5x_8&     +\frac{1}{4}x_2^2x_6^2+\frac{1}{4}x_1^2x_7^2& \\
    & & +x_2x_5x_6x_8-x_2x_4x_8^2& +x_3x_4x_5x_8+x_2x_5x_6x_8-x_2x_4x_8^2&\\
    & & +\frac{1}{2}x_0x_1x_7x_9+\frac{1}{4}x_0^2x_9^2&+\frac{1}{2}x_0x_1x_7x_9+\frac{1}{4}x_0^2x_9^2&\\
    & & +x_1x_7x_{10}+x_0x-9x_{10}+x_{10}^2&    +x_1x_7x_{10}+x_0x_9x_{10}+x_{10}^2& \\ 
    \hline
     \end{array}$
    }
\\ \hspace*{\fill}\\
 \item $\mathfrak{g}_2$: We label the Heisenberg's root by 
$ \begin{array}{|cc|}\hline
      E[0]&\alpha_2  \\
      \hline
      E[1]&\alpha_1+\alpha_2\\
      \hline
      E[2]&3\alpha[1]+\alpha[2]\\
      \hline
      E[3]&2\alpha[1]+\alpha[2]\\
      \hline
      E[4]&3\alpha[1]+2\alpha[2]\\
      \hline
 \end{array}$
\\ \hspace*{\fill}\\ For a generic element $x=x_0E[0]+x_1E[1]+x_2E[2]+x_3E[3]+x_4E[4]$ inside the Heisenberg, we summerize the Hamiltonian in the table:  \\ \hspace*{\fill}\\
 $\begin{array}{|ccc|}
    \hline
    &\varpi_1&\varpi_2\\
    \hline
    \text{Order-1} &\frac{1}{2}x_0x_2-\frac{1}{2}x_1x_3+x_4&2x_4\\
    \hline
    \text{Order-2} &0&-x_1^3x_2-\frac{1}{4}x_0^2x_2^2+\frac{3}{2}x_0x_1x_2x_3+\frac{3}{4}x_1^2x_3^2-x_0x_3^3+x_4^2\\
    \hline\end{array}$
    \end{itemize}
    \hspace*{\fill}\\
   \end{eg}

\subsection{Proof of main theorem}
Now, we can finish the proof of the main theorem \ref{998}. \begin{proof}[Proof of Theorem \ref{998}]\label{ppl}
For part one of the theorem, take the Manin triple for simple algebraic group $G$ to be $\left(G\left(\!\left(t^{-1}\right)\!\right), G_1 [\![ t^{-1}]\!], G[t]\right)$.  From proposition 2.2, theorem 2.3 and corollary 2.9 in \cite{LY}, we deduce that $\overline{\mathcal{W}^\theta_0}$ is a Poisson subvariety of  $\operatorname{Gr}_G$, with nature Poisson structure  invariant under the $G\left(\!\left(t^{-1}\right)\!\right)$-action.

It follows from proposition \ref{2.10} and the isomorphism in \ref{lem2.7} that the Kirillov-Kostant Poisson structure on minimal orbit is the same as the nature $G\left(\!\left(t^{-1}\right)\!\right)$-invariant Poission structure on $\overline{\mathcal{W}^\theta_0}$ induced from the Manin triple.  
\\
For part two, the expression of amiltonians follows from proposition \ref{145}.   It remains to calculate the numbers $m_k$. The method is case by case discussion, we show type $F_4$ and $E_8$ as follows.  The arguments of other types are the same. \\ 
For type $F_4$, identify
$\mathfrak{h}^*=\mathbb{R}^{4}$, $\varepsilon_i$ $(1\leq i\leq 4) $ be the unit coordinate vectors.
We refer to appendix in \cite{bou} for the choice of simple roots, they are
\[
\alpha_1=\varepsilon_2-\varepsilon_3, \alpha_2=\varepsilon_3-\varepsilon_4, \alpha_3=\varepsilon_1, \alpha_4=\frac{1}{2}\left(\varepsilon_1-\varepsilon_2-\varepsilon_3-\varepsilon_4\right).\]

       Highest root: $\check{\theta}=\varepsilon_1+\varepsilon_2=2 \alpha_1+3 \alpha_2+4 \alpha_3+2 \alpha_1.$

    Fundamental weights:
$$
\begin{aligned}
& \varpi_1=\varepsilon_1+\varepsilon_2=2 \alpha_1+3 \alpha_2+4 \alpha_3+2 \alpha_4 \\
& \varpi_2=2 \varepsilon_1+\varepsilon_2+\varepsilon_3=3 \alpha_1+6 \alpha_2+8 \alpha_3+4 \alpha_4 \\
& \varpi_3=\frac{1}{2}\left(3 \varepsilon_1+\varepsilon_2+\varepsilon_3+\varepsilon_4\right)=2 \alpha_1+4 \alpha_2+6 \alpha_3+3 \alpha_4 \\
& \varpi_4=\varepsilon_1=\alpha_1+2 \alpha_2+3 \alpha_3+2 \alpha_1 .
\end{aligned}
$$    

    From the expression of Hamiltonians \eqref{778}, the numbers $m_k$ are essentially the same as the maximal number $r_i$ such that $f_\theta ^{r_i} v_i\neq 0$. Notice that this is equivalent to the maximal number $r_i$ such that $\varpi_i-r_i\check{\theta}$ is a weight of the fundamental representation with highest weight $\varpi_i$.  We fix the usual partial order on weights with $\mu\leq \lambda$ if and only if $\lambda-\mu=\sum_ic_i\alpha_i$ with $c_i\in\Z_{\geq0}$. 
    
    For $i=1$, $\varpi_1-2\check{\theta}=-\varpi_1=w_0\varpi_1$ gives $r_1=2$. 
    For $i=2$, $\varpi_2-3\check{\theta}=-3\alpha_1-3\alpha_2-4\alpha_3-2\alpha_4$, compare the order: $\varpi_2\geq\varpi_2-3\check{\theta}\geq-\varpi_2$,  and $\varpi_2-4\check{\theta}\leq-\varpi_2$ gives $r_2=3$.
    Similar discussion shows $r_3=2$, $r_4=1$.  Labeling these numbers on Dynkin diagram gives the diagram
    \begin{align*}\underset{\mathclap{2}}{\circ} - \underset{\mathclap{3}}{\circ} \Rightarrow \underset{\mathclap{2}}{\circ} - \underset{\mathclap{1}}{\circ}. \end{align*}
\\
For type $E_8$,  we identify 
$\mathfrak{h}^*=\mathbb{R}^8$.
Choose simple roots as in \cite{bou}: 
$$
\begin{gathered}
\alpha_1=\frac{1}{2}\left(\varepsilon_1+\varepsilon_8\right)-\frac{1}{2}\left(\varepsilon_2+\varepsilon_3+\varepsilon_4+\varepsilon_5+\varepsilon_6+\varepsilon_7\right), \\
\alpha_2=\varepsilon_1+\varepsilon_2, \alpha_3=\varepsilon_2-\varepsilon_1, \alpha_4=\varepsilon_3-\varepsilon_2, \alpha_5=\varepsilon_4-\varepsilon_3, \\
\alpha_6=\varepsilon_5-\varepsilon_4, \alpha_7=\varepsilon_6-\varepsilon_5,
\alpha_8=\varepsilon_7-\varepsilon_6.
\end{gathered}
$$

Highest root $\check{\theta}=\varepsilon_7+\varepsilon_8=2 \alpha_1+3 \alpha_2+4 \alpha_3+6 \alpha_4+5 \alpha_5+4 \alpha_6+3 \alpha_7+2 \alpha_8.$\\

The first two fundamental weights are:
\begin{align*}
\varpi_1 & =2 \varepsilon_8=4 \alpha_1+5 \alpha_2+7 \alpha_3+10 \alpha_4+8 \alpha_5+6 \alpha_6+4 \alpha_7+2 \alpha_8 \\
\varpi_2 & =\frac{1}{2}\left(\varepsilon_1+\varepsilon_2+\varepsilon_3+\varepsilon_4+\varepsilon_5+\varepsilon_6+\varepsilon_7+5 \varepsilon_8\right) \\
& =5\alpha_1+8\alpha_2+10 \alpha_3+15 \alpha_4+12 \alpha_5+9 \alpha_6+6 \alpha_7+3 \alpha_8.
\end{align*}
We have   $-\varpi_1=w_0\varpi_1\leq\varpi_1-2\check{\theta}\leq\varpi_1$ and $\varpi_1-3\check{\theta}\leq-\varpi_1$, which give $r_1=2$.  Besides, $-\varpi_2\leq\varpi_2-3\check{\theta}\leq \varpi_2$ and $\varpi_2-4\check{\theta}<-\varpi_2$ gives $r_2=3$. A routine calculation shows $r_3=4$, $r_4=6$, $r_5=5$, $r_6=4$, $r_7=3$, $r_8=2$. 
Labeling these numbers on Dynkin diagram gives the diagram then we obtain:
\begin{align*}\underset{\mathclap{2}}{\circ} - \underset{\mathclap{3}}{\circ} - \underset{\mathclap{4}}{\circ} - \underset{\mathclap{5}}{\circ} - \underset{\mathclap{6}}{\overset{\overset{\textstyle\circ_{\mathrlap{3}}}{\textstyle\vert}}{\circ}} - \underset{\mathclap{4}}{\circ} - \underset{\mathclap{2}}{\circ} .
\end{align*}
For other types, similar discussions give the diagrams \eqref{xiaoxiao} in Figure 1.

    \end{proof}

\section{Further Discussions}
\label{1111}

\subsection{Classical Types and $\mathfrak{g}_2$ }\label{345}

In this subsection, we compute Hamiltonians for type $ABCD$ and $G_2$ in more explicit form. Fix a complex simple classical Lie algebra $\mathfrak{g}$, consider an element in its minimal nilpotent orbit: $A=\left(a_{i j}\right)\in\mathcal{O}_{\text{min}}$, let  $A_{k, k} \text { be the left-upper }k\times k \,\text{submatrix of } A \text {. }$ Then from part two of theorem \ref {998}, proposition \ref{999} and remark \ref{1000}, we have:
\begin{thm}\label{3.1}
The Hamiltonians of the integrable system on $\overline{\mathcal{O}}_{\text{min}}$ for each $\mathfrak{g}$ are giving  by the follwing:

\begin{enumerate}
\item Type $A_n$,  let $A=(a_{ij})_{1\leq i,j\leq n+1}\in \mathfrak{sl}(n+1,\mathbb{C}),$\,there are $n=h^\vee-1$ Hamiltonians: $\mathrm{Tr}A_{k,k}$, for $1\leq k\leq n$. 
\item Type $B_n$, let  $A=(a_{ij})_{1\leq i,j\leq 2n+1}\in\mathfrak{so}(2n+1,\mathbb{C})$, there are $2n-2=h^\vee-1$ Hamiltonians:  $\operatorname{Tr} (A_{k, k})$ for $1\leq k\leq n$, $\, \operatorname{Tr}\left(\bigwedge^2 A_{k, k}\right)$ for $2\leq k\leq n-1$, 
\item Type $C_n$, let $A=(a_{ij})_{1\leq i, j\leq 2n}\in \mathfrak{sp}(2n,\mathbb{C})$, there are $n=h^\vee-1$ Hamiltonians: $\mathrm{Tr}(A_{k,k})$ for $1\leq k\leq n$.
\item Type $D_n$, let $A=(a_{ij})_{1\leq i, j\leq 2n}\in\mathfrak{so}(2n,\mathbb{C})$, there are $2n-3=h^\vee-1$ Hamiltonians: $\operatorname{Tr}( A_{k, k})$ for $1\leq k\leq n$, $\operatorname{Tr}\left(\bigwedge^2 A_{k, k}\right)$ for $2\leq k\leq n-2$.
\end{enumerate}
\end{thm}

We remark that theorem \ref{3.1} can also be calculated directly for each type. We discuss type $A$, $C$ and $D$ in the following. \\
\textbf{Type} $\mathbf{A_n}$: Let $\mathfrak{g}=\mathfrak{sl}(n+1,\mathbb{C})$, $G=SL(n+1)$, the minimal orbit is the orbit of maximal root vector $e_{\theta^\vee}$ under the adjoint action, i.e, 
$\mathcal{O}_{\min}=G\cdot e_\theta=\{g\cdot \begin{pmatrix}
0 & \cdots & 0 & 1 \\
 \vdots& \ddots & & 0 \\
0& &\ddots & \vdots \\
0 & & \cdots & 0
\end{pmatrix} g^{-1}\}_{g\in G}$.
Let $A \in \mathcal{O}_{\text {min }}$, then $A^2=0$, rank $A=1$. Write $A=x y^t=\left(x_i y_j\right)$, where $x=\left(\begin{array}{c}x_1 \\ \vdots \\ x_{n+1}\end{array}\right), y=\left(\begin{array}{c}y_1 \\ \vdots \\ y_{n+1}\end{array}\right)$
, then $ \exp (t^{-1} A)=1+t^{-1} A$ . For $1\leq k\leq n+1$, the $k$-th fundamental representation is $V_{\varpi_k}=\bigwedge^k\mathbb{C}^{n+1}$, its highest weight vector is $e_1 \wedge \cdots \wedge e_k$. \\
\[\begin{aligned} \exp \left(t^{-1} A\right)\left(e_1 \wedge \cdots \wedge e_k\right) & =\left(1+t^{-1} A\right)\left(e_1 \wedge \cdots \wedge e_k\right)\\&=\left(e_1+t^{-1} y_1 x\right) \wedge \cdots \wedge\left(e_k+t^{-1} y_k x\right)\\&=e_1 \wedge \cdots \wedge e_k+t^{-1} \sum_{i=1}^k y_i \cdot e_1 \wedge \cdots \wedge x\wedge \cdots \wedge e_k+\cdots\\& =\left(1+t^{-1} \sum_{i=1}^k y_i x_i\right)\left(e_1 \wedge \cdots \wedge e_k\right)+\text { higher terms }\end{aligned}\]The higher terms vanish by linear dependence, hence we have
$\left\langle v^k, \exp \left(t^{-1} x\right) v_k\right\rangle=1+t^{-1} \sum_{i=1}^k y_i x_i=1+t^{-1}\mathrm{Tr}A_{kk}$. 
Since $\mathrm{Tr}(A_{n+1,n+1})=0$, we obtain all the Hamiltonians are 
$\langle(e_1\wedge e_2\cdots \wedge e_k)^*,\exp(t^{-1}A)e_1\wedge e_2\wedge\cdots\wedge e_k\rangle= \mathrm{Tr}(A_{k,k})$ for $1\leq k\leq n$. Observe they are exactly the linear term of Gelfand Zetlin system.\\

\textbf{Type} $\mathbf{C_n}$: Let $\mathfrak{g}=\mathfrak{sp}(2n,\mathbb{C})$, root space of maximal root $\mathfrak{g}_{\theta^\vee}=\mathbb{C}E_{1,n+1}$. For $1\leq k\leq n$, highest weight vector of the fundamental representations $V_{\varpi_k}$ are $v_k=e_1\wedge\cdots \wedge e_{k-1}\wedge e_{k}$, see \cite[pp.261]{fulton13}.
Let $A \in \mathcal{O}_{\text {min }}$, then $A^2=0$, rank $A=1, A=x y^t=\left(x_i y_j\right)$
Similar as type $A$, we have 
$\left\langle v^k, \exp \left(t^{-1} x\right) v_k\right\rangle=1+t^{-1} (\sum_{i=1}^{k} y_i x_i)=1+t^{-1}(\mathrm{Tr}A_{k,k})$. Therefore the Hamiltonians are  $\mathrm{Tr}(A_{k,k})$ for $1\leq k\leq n$. \\

\textbf{Type} $\mathbf{G_2}$:
Select a basis of $\mathfrak{g}_2$ by $\{H_1,H_2, X_1,X_2,\cdots, X_6,Y_1,Y_2 \cdots Y_6\}$ as in the convention 
in \cite[pp.346]{fulton13}, note that  each of the brackets $\left[X_i, Y_i\right]$ will be the distinguished element of $\mathfrak{h}$ corresponding to the root $\alpha_i$, and the highest root vector is $e_\theta=X_6$. From proposition \ref{xiaobaobao}, we know there are three Hamiltonians: two linear of form $\varpi_1$, $\varpi_2$ and one second order adjoint Hamiltonian associated with the fundamental weight $\varpi_2=3\alpha_1+2\alpha_2$.  In this case, $i=2$, $R_{P_2}=\alpha_1$, $R-R_{P_2}= \{\alpha_2, \alpha_3, \alpha_4, \alpha_5,\alpha_6\}$. 
Note that we have 
\begin{align*}
&H_3=H_1+3 H_2, \quad H_4=2 H_1+3 H_2 \\&
H_5=H_1+H_2, \quad  H_6=H_1+2 H_2
\end{align*}

Therefore
\begin{align*}&
\langle \varpi_2,H_{\alpha_1}\rangle=0, \quad \langle \varpi_2,H_{\alpha_2}\rangle=1, \quad \langle \varpi_2,H_{\alpha_3}\rangle=3\\&
\langle \varpi_2,H_{\alpha_4}\rangle=3,\quad \langle \varpi_2,H_{\alpha_5}\rangle=1,\quad
\langle \varpi_2,H_{\alpha_6}\rangle=2.
\end{align*}
Hence by the second part of proposition \ref{xiaobaobao}, this Hamiltonian is of the following form:
\begin{align*}
{}& \omega_2^2 +\langle \omega_2,\alpha_2^\vee\rangle X_2^\vee Y_2^\vee+\langle \alpha_3^\vee,\omega_2\rangle X_3^\vee Y_3^\vee+\langle \alpha_4^\vee,\omega_2\rangle X_4^\vee Y_4^\vee+\langle \alpha_5^\vee,\omega_2\rangle X_5^\vee Y_5^\vee+\langle \alpha_6^\vee,\omega_2\rangle X_6^\vee Y_6^\vee \\
&=\omega_2^2+X_2^\vee Y_2^\vee+3X_3^\vee Y_3^\vee+3X_4^\vee Y_4^\vee+X_5^\vee Y_5^\vee+2X^\vee_6Y^\vee_6
\end{align*}

We mention that there is another way to calculate the second order Hamiltonian directly. 
Denote an element $A$ in the minimal nilpotent cone by \[A=h_1 H_1+h_2 H_2+x_1 X_1+\cdots +x_6 X_6+y_1 Y_1+\cdots+ y_6Y_6.\] 
Note that for $X,Y$ in Lie algebra,
\[\operatorname{Ad}_{\exp X}(Y)=\exp \left(\operatorname{ad}_X\right)(Y)=Y+[X, Y]+\frac{1}{2 !}[X,[X, Y]]+\frac{1}{3 !}[X,[X,[X, Y]]]+\cdots\] 
Hence to compute the second order term of $\langle\mathrm{exp}(t^{-1}A)e_\theta, e_\theta^*\rangle$, it suffices to compute $([A,[A,X_6]],X_6^*)$.
Using the multiplication table in \cite{fulton13}, we have:
\[[A,X_6]=h_2X_6-y_2X_5-y_3X_4+y_4X_3+y_5X_2-y_6(H_1+2H_2)\]

\begin{equation}
\begin{aligned}
{}&
\langle [A, X_5],X_6^*\rangle =-x_2\quad
\langle [A, X_4],X_6^*\rangle =-3x_3\quad
\langle [A, X_3],X_6^*\rangle =3x_4\\&
\langle [A, X_2],X_6^*\rangle =x_5\quad
\langle [A, H_1],X_6^*\rangle =0\quad
\langle [A, H_2],X_6^*\rangle =-x_6
\end{aligned}
\end{equation}Therefore the desired second order Hamiltonian is:
\begin{equation}
\begin{aligned}
\langle[A,[A,X_6]],X_6^*\rangle={}&\langle h_2[A,X_6]-y_2[A,X_5]-y_3[A,X_4]+y_4[A,X_3]+y_5[A,X_2]-y_6[A,H_1+2H_2],X_6^*\rangle\\&=h_2^2+x_2y_2+3x_3y_3+3x_4y_4+x_5y_5+2x_6y_6.
\end{aligned}
\end{equation}

\subsection{Discussions for quantization}

There are two guesses of quantization for the Hamiltonians, which means lifting them to a noncommutative algebra, such that after lifting they  pairwise commute under the usual commutator. 

\textbf{Travkin's quantization.}
Let $\widetilde{Cas}:=\sum_{i=1}^{l} \frac{1}{\kappa(\alpha_i,\alpha_i)}e_i\otimes e_i\in U(\mathfrak{g})\otimes U(\mathfrak{g})$, $\kappa$ be the Killing form. 
 Roman Travkin conjectures that the quantization of the functions  $\langle v^m,x^nv_m\rangle$ in $\C[\mathfrak{g}]$ may be 
$\widetilde{\mathrm{Cas}}_{(2)}^n\langle\widetilde{\mathrm{Cas}}_{(1)}^n v_m, v^m\rangle$ in the universal envelopping algebra $U(\mathfrak{g})$,

In particular,  the degree $n=1$ quantization are Cartan $H_i$, the degree 2 quantization associated to fundamental weight $\omega_i$ are 
\[H_i^2+\sum_{\alpha\in R^+-R_{P_i}}\frac{\langle \omega_i,\alpha^\vee\rangle}{\kappa(\alpha,\alpha)} E_\alpha E_{-\alpha}\in U(\mathfrak{g})/\mathcal{J},\]where $\mathcal{J}$ is the Joseph's ideal.  The degree-1 $H_i$ certainly commutes with each other, and the degree one also comuutes with degree-2 ,which following from following computation:

\begin{align*}
H_\beta E_\alpha E_{-\alpha}&=(\langle \beta^\vee,\alpha\rangle E_\alpha+E_\alpha H_\beta)E_{-\alpha}\\&=\langle\beta^\vee,\alpha\rangle E_\alpha E_{-\alpha}+E_\alpha(\langle \beta^\vee,-\alpha\rangle E_{-\alpha}+E_{-\alpha}H_\beta)\\&=E_\alpha E_{-\alpha}H_\beta.
\end{align*}

Hence $[H_\beta,E_\alpha E_{-\alpha}]=0$,  i.e, $[ H_j, H_i^2+\sum_{\alpha\in R^+-R_{P_i}}\frac{\langle \omega_i,\alpha^\vee\rangle}{\kappa(\alpha,\alpha)} E_\alpha E_{-\alpha}]=0$. 
For higher degrees, one may seek to verify if these commutators falls inside Joseph's ideal $\mathcal{J}$ \cite{Jos} , but this is overwhelming since the generators of $\mathcal{J}$ are difficult to write explicitly. 

\textbf{Heisenberg orbit quantization.}
This proposal conjectures that the quantization of the $r$-th matrix coefficient function associated to fundamental weight $\varpi_i$ be 

\[\sum_{k=1}^{l}\frac{1}{k!}\widetilde{\mathrm{Kos}}^k_{(2)}\langle \widetilde{\mathrm{Kos}}^k_{(1)}f_\theta^r v_{\varpi_i},v^{\varpi_i}\rangle\in U(n^\star)\]

where $\widetilde{\operatorname{Kos}}=\sum_{i=1}^{2h^\vee-3}\frac{1}{\kappa(\varphi_i,\varphi_i)}E_i\otimes E_i\in U(\mathfrak{n}^\star)\otimes U(\mathfrak{n}^\star)$ (see the notations in \ref{778}).

\appendix
\section{Calculation in matrix expression for type $B$,$D$}
\subsection{Type D}
\textbf{Type} $\mathbf{D_n}$: Let $\mathfrak{g}=\mathfrak{so}(2n,\mathbb{C})=\mathfrak{so}(V)$, $V=\mathrm{span}\{e_1,\cdots, e_{2n}\}$. The Cartan subalgebra $\mathfrak{h}$ is chosen to be  generated by the $n$   $2 n \times 2 n$ matrices $H_i=E_{i, i}-$ $E_{n+i, n+i}$ . Take as basis for the dual vector space $\mathfrak{h}^*$ the dual basis $L_j$, where $\left\langle L_j, H_i\right\rangle=\delta_{i, j}$.
Take simple roots $\alpha_i=L_i-L_{i+1}$ for $1 \leq i \leq n-1$, and $\alpha_n=L_{n-1}+L_n$.  The fundamental weights $\varpi_i$ are put in the order: $\varpi_i=L_1+\cdots+L_i$ for $1 \leq i \leq n-2$, and
$$
\varpi_{n-1}=\left(L_1+\cdots+L_{n-1}-L_n\right) / 2, \quad \varpi_n=\left(L_1+\cdots+L_n\right) / 2
$$

we have $\varpi_i\left(H_{\alpha_j}\right)=\delta_{i, j}$.

The $\text { highest root } \theta=(1,1,0, \ldots, 0)$, the associated root vector is $\mathfrak{g}_\theta=\mathbb{C} \cdot(E_{1, n+2}-E_{2, n+1})$.
Let $A \in \mathcal{O}_{\min }$ , then $A^2=0$. 

When $1 \leq k \leq n-2$,
The highest weight of fundamental representation $V_{\varpi_k}$ is $v_k=e_1 \wedge \cdots \wedge e_{k-1} \wedge e_{k}$ (see \cite[pp.288]{fulton13}).
Then \[\begin{aligned} \left(1+A t^{-1}\right)\left(e_1 \wedge \cdots \wedge e_{k-1} \wedge e_{k}\right)={}&\left(e_1+t^{-1} A e_1\right) \wedge \cdots \wedge\left(e_{k}+t^{-1} A e_{k}\right) \\ &=e_1 \wedge \cdots \wedge e_{k}+t^{-1}\sum_{i=1}^{k}\left(e_1 \wedge \cdots \wedge A e_i \wedge \cdots \wedge e_{k}\right)\\&+t^{-2} \sum_{1\leq i<j\leq k} e_1 \wedge \cdots \wedge A e_i \wedge \cdots \wedge Ae_j\wedge \cdots \wedge e_{k} \end{aligned}\]
$\text { The higher terms vanish since } \mathrm{rank}\,A\leq 2 . $ Since we will pair the above with $v^k=(e_1\wedge\cdots\wedge e_{k-1}\wedge e_{k})^*$, we need only calculate the coefficient of $e_1 \wedge \cdots \wedge e_{k-1} \wedge e_{k}$ in the above term.
Notice that\[ Ae_i\wedge Ae_j= \begin{pmatrix}a_{1i} \\ \vdots\\ a_{ii}\\ \vdots \\ a_{ji}\\ \vdots \\ a_{2n,i}\end{pmatrix}    \bigwedge    \begin{pmatrix}a_{1j} \\ \vdots\\ a_{ij}\\ \vdots \\ a_{jj}\\ \vdots \\ a_{2n,j}\end{pmatrix}     \]and the coefficient of its $e_i\wedge e_j$ term is $\left|\begin{array}{ll}a_{i i} & a_{i j} \\ a_{j i} & a_{j j}\end{array}\right|$, so we arrive at

\begin{equation}
\langle v^k, \exp (t^{-1} x) v_k\rangle=
1+t^{-1} (\sum_{i=1}^{k} a_{i i})+\\t^{-2} \sum_{1\leq i<j\leq k}\left|\begin{array}{ll}a_{i i} & a_{i j} \\ a_{j i} & a_{j j}\end{array}\right| e_1 \wedge \cdots \wedge e_i \wedge \cdots \wedge e_j \wedge \cdots\wedge e_{k-1} \wedge e_{k}
\\
\end{equation}

Which is equal to the following \[1+t^{-1} (\operatorname{Tr} A_{k, k})+t^{-2} (\operatorname{Tr}\left(\bigwedge^2 A_{k, k}\right))\]

When $k=n-1, n$ , the fundamental representations $V_{\varpi_k}$ are the Spinor representations. In the following we use clifford algebra representation to calculate the corresponding hamiltonians.
\begin{define}Denote the bilinear form $Q$ by 
\[\quad Q(x, y)=x^t\left(\begin{array}{cc} 0 & I_n \\ I_n &0 \end{array}\right)    y ,\]Then $V$ can be splited into direct sum of isotropic subspaces: \[
 V=W \oplus W^{\prime}\]
Denote $C(Q)$ the Clifford algebra associated to $Q$ and $V$. We regard it as a Lie algebra, where the Lie bracket is induced from the standard commutator $[,]$ on the associated algebra structure.  For any $ w \in W$  we denote operator  $L_w \in \operatorname{End}(\dot{\bigwedge}W),\,L_w: \xi \mapsto \omega \wedge \xi   $. For any $ \theta \in W^*$, define  $D_\theta \in \operatorname{End}(\dot{\bigwedge}W), \, D_\theta: 1 \mapsto 0, \,  
 \omega \mapsto \theta(\omega) \in\mathbb{C} \,\text{ and satisfies the following relation}$ \[\xi \wedge \eta \mapsto D_\theta(\xi) \wedge \eta+(-1)^{\operatorname{deg}(\xi)} \xi \wedge D_\theta(\eta). \]
\end{define}
\begin{lma}We have lie algebra homorphisms: 
 \begin{equation}\label{1}
 \mathfrak { so }(Q) \congto  \bigwedge^2 V \hookrightarrow  C(Q)    \congto  \operatorname{End}(\dot{\bigwedge} W) \end{equation}, where the first and last map are isomorphisms, the middle map is embedding. The map of basis are as following:
\[ E_{i, j}-E_{n+j, n+i} \mapsto \frac{1}{2}(e_i \wedge e_{n+j})  \mapsto \frac{1}{2}(e_i \cdot e_{n+j}-\delta_{i j}) \mapsto\frac{1}{2}(L_{e_i} \circ D_{2 e_j^*}-\delta_{i j} \mathrm{Id}) \]
\[ E_{i, n+j}-E_{j, n+i} \mapsto\frac{1}{2} (e_i\wedge e_j) \mapsto \frac{1}{2} e_i \cdot e_j \longmapsto \frac{1}{2}L_{e_i} \circ L_{e_j}  \]
 \[E_{n+i, j}-E_{n+j, i}\mapsto\frac{1}{2}(e_{n+i} \wedge e_j)\mapsto \frac{1}{2}(e_{n+i}\cdot e_j-\delta_{i j}) \mapsto \frac{1}{2}\left(D_{2 e_i^*} L_{e_j}-\delta_{ij}\mathrm{Id}\right) \]
\end{lma}

Notice that the image of the map $\bigwedge^2 V \hookrightarrow  C(Q)$ are of even order in our circumstance, and there is an isomorphism 

 \[\quad(C Q)^{\text {even }} \cong \mathfrak{gl}\left(\bigwedge^{\text {even }} W\right) \oplus \mathfrak{gl}\left(\bigwedge^{\text {odd }} W\right). \]
 
We can rewrite the homomorphism \eqref{1} as 
\[\mathfrak { so }(Q) \congto  \bigwedge^2 V \hookrightarrow  C(Q)^{\text{even}}\congto\mathfrak{gl}\left(\bigwedge^{\text {even }} W\right) \oplus \mathfrak{gl}\left(\bigwedge^{\text {odd }} W\right)\]The two Spinor representations are exactly the two decompositions $\bigwedge^{\text{even}}W$ and $\bigwedge^{\text{odd}}W$. To be explicit, 
we need discuss separatedly according to the parity of $n$. 
If $n$ is even, \begin{enumerate}\item The first Spinor representation is $V_{\varpi_n}=\bigwedge^{\text {even }} W=\Gamma_{\frac{1}{2}\left(L_1+\cdots+L_n\right)}$,with highest weight vector, $e_1 \wedge \cdots \wedge e_n .$
 \begin{enumerate}
\item
If $n \neq j$, then
$$
\begin{aligned}
 D_{2 e_j^*}\left(e_1 \wedge \cdots \wedge e_n\right)={} &
 D_{2 e_j^*}\left(e_1 \wedge \cdots \wedge e_{j-1} \wedge e_{j+1} \wedge \cdots \wedge e_n\right) \wedge e_j \cdot(-1)^{n-j}\\&+(-1)^{j-1}\left(e_1 \wedge \cdots  e_{j-1} \wedge e_{j+1} \wedge \cdots \wedge e_n\right) \wedge D_{2 e_j^*}\left(e_j\right) \\
= & (D_{2 e_j^*}(e_1 \wedge \cdots  e_{j-1} \wedge e_{j+1} \wedge \cdots \wedge e_{n-1}) \wedge e_n+ \\ &(-1)^{j-1} e_1 \wedge \cdots \wedge e_{j-1} \wedge e_{j+1} \wedge \cdots \wedge e_{n-1} \wedge D_{2 e_j^*}(e_n)))\wedge e_j \cdot(-1)^{n-j}+ \\
& 2 \cdot(-1)^{j-1}(e_1 \wedge \cdots \wedge e_{j-1} \wedge e_{j+1} \wedge \cdots \wedge e_n)\\
= & \cdots \\
= & 2 \cdot(-1)^{j-1}\left(e_1 \wedge \cdots \wedge e_{j-1} \wedge e_{j+1} \wedge \cdots \wedge e_n\right)
\end{aligned}
$$
The above equation follows from $D_{2 e_j^*}\left(e_n\right)=0$.
\begin{enumerate}\item If $i \neq j$, then $L_{e_i} \circ D_{2 e_j^*}\left(e_1 \wedge \cdots \wedge e_n\right)=2 \cdot(-1)^{j-1} \cdot(-1)^{j-1} e_i \wedge\left(e_1 \wedge \cdots \wedge e_{j-1} \wedge e_{j+1} \wedge \cdots \wedge e_n\right)=0$.
\item If $i=j$ then $L_{e_i} \circ D_{2 e_j^*}\left(e_1 \wedge \cdots \wedge e_n\right)=2 \cdot (-1)^{j-1} \cdot(-1)^{j-1} e_1 \wedge \cdots \wedge e_n=2 e_1 \wedge \cdots \wedge e_n$
\end{enumerate}
\item If $n=j \neq i$, then it follows similarly $L_{e_i} \circ D_{2 e_j^*}\left(e_1 \wedge \cdots \wedge e_n\right)=0$

\item If  $i=j=n$ , then  
 \[\begin{aligned}
      L_{e_n} D_{2 e_n^*}\left(e_1 \wedge \cdots \wedge e_n\right) & =L_{e_n}\left(D_{2 e_n^*}\left(e_1 \wedge \cdots \wedge e_{n-1}\right) \wedge e_n+(-1)^{n-1} e_1\wedge\cdots\wedge e_{n-1}\wedge D_{2e_n^*}(e_n))  \right. \\
& =L_{e_n}\left(0+(-1)^{n-1} \cdot 2 \cdot e_1 \wedge \cdots \wedge e_{n-1}\right)
 \\
& =(-1)^{n-1} \cdot(-1)^{n-1} \cdot 2 \cdot e_1 \wedge \cdots \wedge e_{n-1} \wedge e_n \\
& =2 e_1 \wedge \cdots \wedge e_n \\
\end{aligned}\]

Substitute into the equation of Schubert divisor, we have
\begin{align*}
 \left(1+A t^{-1}\right)&\left(e_1 \wedge \cdots \wedge e_n\right)= ( 1+t^{-1}(\sum_{1 \leqslant i, j \leqslant n} a_{i j}\left(E_{i, j}-E_{n+j, n+i}\right)+
 \sum_{1 \leqslant i< j \leqslant n} a_{i j}\left(E_{i, n+j}-E_{j i n+i}\right)\\&+\sum_{1 \leqslant i <j\leq n} a_{i j}\left(E_{n+i, j}-E_{n+j, i}\right)))\left(e_1 \wedge \cdots \wedge e_n\right))\\=&
(1+t^{-1}(\sum_{1\leq i\neq j\leq n}\frac{1}{2}a_{ij}(L_{e_i}\circ D_{2e_j^*}-\delta_{ij}\mathrm{Id})+\sum_{1\leq i\leq n} \frac{1}{2}a_{ii}(L_{e_i}\circ D_{2e_i^*}-\delta_{ii}\mathrm{Id})\\&+\sum_{1\leq i<j\leq n}\frac{1}{2}a_{n+i,j}L_{e_i}\circ L_{e_j} +\sum_{1\leq i<j\leq n}\frac{1}{2}a_{i,j+n}(D_{2e_i^*}\circ L_{e_j}-\delta_{ij}\mathrm{Id}) ))(e_1\wedge\cdots\wedge e_n)\\=&
(1+\frac{1}{2}t^{-1}\sum_{1\leq i\leq n}a_{ii})e_1\wedge\cdots\wedge e_n
\end{align*}
\end{enumerate}
\item
The second Spinor representation is
$V_{\omega_{n-1}}=\Gamma_{\frac{1}{2}\left(L_1+\cdots+L_n-L_n\right)} =\bigwedge^{\text {odd }} W$ with highest weight vector $e_1 \wedge \cdots \wedge e_{n-1}. $

\begin{enumerate}\item If $i , j\leq n-1$, $i\neq j$ ,  then it follows similarly \[L_{e_i} D_{2 e_j^*}\left(e_1 \wedge \cdots \wedge e_{n-1}\right)=0\]
\item If $i=j \leqslant n-1$, then same calculation as previous gives
\[
L_{e_i} \circ D_{2 e_i} *\left(e_1 \wedge \cdots \wedge e_{n-1}\right)=2 e_1 \wedge \cdots \wedge e_{n-1} . \]\item If $j\leq n-1$, $i=n$, then 
 \[L_{e_i} D_{2 e_j^*}\left(e_1 \wedge \cdots \wedge e_{n-1}\right)= 2(-1)^{j-1}(-1)^{n-2}e_1\wedge\cdots \wedge e_{j-1}\wedge e_{j+1}\wedge\cdots\wedge e_n\]
This term will not contribute when pairing with $(e_1\wedge\cdots \wedge e_{n-1})^*$.

\item If $j=n$, then 
\begin{align*}&L_{e n} D_{2 e_n^*}\left(e_1 \wedge \cdots \wedge e_{n-1}\right)\\&=L_{e_n}\left(D_{2 e_n^*}\left(e_1 \wedge \cdots \wedge e_{n-2}\right) \wedge e_{n-1}+(-1)^{n-1}\left(e_1 \wedge \cdots \wedge e_{n-2}\right) \wedge D_{2 e_n^*}\left(e_{n-1}\right)\right)\\
& =L_{e_n}\left(D_{2 e_n^*}\left(e_1 \wedge \cdots \wedge e_{n-2}\right) \wedge e_{n-1}\right) \\
& =L_{e_n}\left(D_{2 e_n^*}\left(e_1 \wedge \cdots \wedge e_{n-3}\right) \wedge e_{n-2} \wedge e_{n-1}+(-1)^{n-2}\left(e_1 \wedge \cdots \wedge e_{n-3}\right) \wedge D_{2 e_n^*}\left(e_{n-2}\right) \wedge e_{n-1}\right) \\
& =\cdots=L_{e_n}\left(D_{2 e_n^*}\left(e_1\right)\right)=0 \end{align*}
Calculation as before yields
\[
\left\langle\left[1+\left(A t^{-1}\right)\right)\left(e_1 \wedge \cdots \wedge e_{n-1}\right),\left(e_1 \wedge \cdots \wedge e_{n-1}\right)^*\right\rangle=1+\frac{1}{2} \mathrm{Tr} A_{n-1, n-1} t^{-1}
\]

\end{enumerate}
\end{enumerate}
If $n$ is odd,
\[
V_{\varpi_{n-1}}=\bigwedge ^{\text{even}}W=\Gamma_{\frac{1}{2}\left(L_1+\cdots+L_{n-1}-L_n\right)} ,\quad V_{\varpi_n}=\bigwedge ^{\text{odd}}W=\Gamma_{\frac{1}{2}\left(L_1+\cdots+L_n\right)}
\]
The corresponding Hamiltonians are exactly the same as $n$ is odd.
In conclusion, the Hamiltonians are  for $\mathrm{Tr}(A_{kk})$ and $\operatorname{Tr}\left(\bigwedge^2 A_{k, k}\right)$ for $1\leq k\leq n-2$, and $\mathrm{Tr}(A_{n-1,n-1})$ ,$\mathrm{Tr}(A_{n,n})$ for Spinor representations. Notice that $\operatorname{Tr}\left(\bigwedge^2 A_{1, 1}\right)=0$, so there are actually $2n-3=h^\vee-1$ Hamiltonians.

\subsection{Type B}

Denote $\mathfrak{g}=\mathfrak{so}(2n+1,\mathbb{C})=\mathfrak{so}(V)$,  where $V$ is a complex vector space of dimension $2n+1$, and the standard bilinear form $Q$ on $V$ by \[\quad Q(x, y)=x^t\left(\begin{array}{ccc} 0 & I_n&0 \\ I_n &0&0\\0&0&1 \end{array}\right)y. \]
Decompose $V$ by
\[V=W\oplus W'\oplus U\]where $W$, $W'$ are isotropic subspace of dimension $n$, $U$ is of dimension 1. 
Elements in $\mathfrak{g}$ can be regarded as matrix block form 
$X=\left(\begin{array}{c|c|c}
A & B & E \\
\hline C & D & F \\
\hline G & H & 0
\end{array}\right)$, where 

$B$ and $C$ are skew-symmetric $n\times n$ matrices and $A$ and $D$ negative transposes of each other; and in addition $E=-{ }^t H, F=-{ }^t G$.

We take the Cartan subalgebra to be the subalgebra of  diagonal matrices in this representation, i.e, $\mathfrak{h}$ is generated by the $n$ matrices of shape $(2n+1)\times (2n+1)$,  $H_i=E_{i,i}-E_{n+i,n+i}$.
We will correspondingly take as basis for the dual vector space $\mathfrak{h}^*$ the dual basis $L_j$, where $\left\langle L_j, H_i\right\rangle=\delta_{i, j}$
Denote the maximal root $\theta=L_1+L_2$, the minimal nilpotent orbit is the conjugation of the root vector $\mathfrak{g}_\theta=\mathbb{C}(E_{1,n+2}-E_{2,n+1})$, so an element $A$ in the orbit is of rank 2 and $A^2=0$. 
For $k=1, \ldots, n-1$, the exterior power $\bigwedge^k V$ of the standard representation $V$ of $\mathfrak{s o}_{2 n+1} \mathbb{C}$ is the fundamental representation with highest weight $\varpi_k=L_1+\cdots+L_k$, highest weight vector is $e_1\wedge\cdots\wedge e_k$.
\[\begin{aligned} \left(1+A t^{-1}\right)\left(e_1 \wedge e_2\wedge \cdots \wedge e_{k-1} \wedge e_{k}\right){}&=\bigwedge_{i\in\{1,\cdots, k\}}\left(e_i+t^{-1} A e_i\right) \\ &=e_1 \wedge \cdots \wedge e_{k-1} \wedge e_{k}+t^{-1}\sum_{i=1}^{k}\left(e_1 \wedge \cdots \wedge A e_i \wedge \cdots \wedge e_{k}\right)\\&+t^{-2} \sum_{1\leq i<j\leq k} e_1 \wedge \cdots \wedge A e_i \wedge \cdots \wedge Ae_j\wedge \cdots \wedge e_{k} \end{aligned}\]After pairing the above term with $v^k=(e_1\wedge\cdots\wedge e_{k-1}\wedge e_{k})^*$, we get \[1+t^{-1} (\sum_{i=1}^{k} a_{i i})+t^{-2} \sum_{1\leq i<j\leq k}\left|\begin{array}{ll}a_{ii} & a_{ij} \\ a_{ji} & a_{jj}\end{array}\right| e_1 \wedge \cdots \wedge e_i \wedge \cdots \wedge e_j \wedge \cdots \wedge e_{k}\]Which equals to
\begin{equation}
\begin{split}
&1+t^{-1}( \operatorname{Tr} A_{k, k})+ t^{-2}( \operatorname{Tr}\left(\bigwedge^2 A_{k, k}\right))
\end{split}
\end{equation}
There is one Spinor representation left for $k=n+1$. We still have following isomorphism as associated algebras:

\[C(Q) \cong \operatorname{End}\left(\dot{\bigwedge} W\right) \oplus \operatorname{End}\left(\dot{\bigwedge} W^{\prime}\right)\]\[C(Q)^{\text {even }} \cong \operatorname{End}\left(\dot{\bigwedge} W\right)\]and get the Spinor representation by the embedding into Clifford algebra:

\[\mathfrak{s o}_{2 n+1} \mathbb{C}=\mathfrak{s o}(Q) \hookrightarrow C(Q)^{\text {even }} \cong \mathfrak{g l}\left(\dot{\bigwedge} W\right)\]
The representation $\dot{\bigwedge} W$ is the fundamental representation of $\mathfrak{s o}_{2 n+1} \mathbb{C}$ with highest weight
$$
\alpha=\frac{1}{2}\left(L_1+\cdots+L_n\right) .
$$

We write a basis of the Lie algebra by
$
X_{i, j}=E_{i, j}-E_{n+j, n+i} ,
Y_{i, j}=E_{i, n+j}-E_{j, n+i},
Z_{i, j}=E_{n+i, j}-E_{n+j, i},
U_i=E_{i, 2 n+1}-E_{2 n+1, n+i}
$,
and
$
V_i=E_{n+i, 2 n+1}-E_{2 n+1, i}
.$
We already know the action of $X_{i,j}$, $Y_{i,j}$, $Z_{i,j}$ as in the calculation of type D, now we determine the action of $U_i$ and $V_i$. 
\begin{lma}
The basis $U_i$ in the Lie algebra incorresponds the operator  $L_{e_i}\circ (\mathrm{Id}|_{\bigwedge^{\text{even}}W}-\mathrm{Id}|_{\bigwedge^{\text{odd}}W})$ in the representation $\dot{\bigwedge} W$  , while $V_i$ is the operator $D_{2e_i^*}\circ (\mathrm{Id}|_{\bigwedge^{\text{even}}W}-\mathrm{Id}|_{\bigwedge^{\text{odd}} W})$.
\end{lma}
\begin{proof}
Recall the isomorphism 

\[\bigwedge^2 V \cong \mathfrak{so}(Q) \subset \operatorname{End}(V)\]
is given by
\[a \wedge b \mapsto \varphi_{a \wedge b}\]
for $a$ and $b$ in $V$, where $\varphi_{a \wedge b}$ is defined by
$$
\varphi_{a \wedge b}(v)=2(Q(b, v) a-Q(a, v) b)
$$

We have \[\varphi_{e_i\wedge e_{2n+1}}(v)=2(Q(e_{2n+1},v)\cdot e_i-Q(e_i,v)\cdot e_{2n+1})=2(v_{2n+1}\cdot e_i-v_{n+i}\cdot e_{2n+1})\]
Hence $\varphi_{e_i\wedge e_{2n+1}}=2(E_{i,2n+1}-E_{2n+1,n+i})$. That is, the basis $U_i$ in Lie algebra maps via Clifford algebra into $\mathfrak{g l}\left(\dot{\bigwedge} W\right)$ 
as following
\[e_i\wedge e_{2n+1}\mapsto 2(E_{i,2n+1}-E_{2n+1,n+i})\mapsto e_i\cdot e_{2n+1}\mapsto  L_{e_i}\circ (\mathrm{Id}|_{\bigwedge^{\text{even}}W}-\mathrm{Id}|_{\bigwedge^{\text{odd}}W}).\]
Similarly,
\[e_{n+i}\wedge e_{2n+1}\mapsto 2(E_{n+i,2n+1}-E_{2n+1,i})\mapsto e_{n+i}\cdot e_{2n+1}\mapsto D_{2e_i^*}\circ (\mathrm{Id}|_{\bigwedge^{\text{even}}W}-\mathrm{Id}|_{\bigwedge^{\text{odd}} W}).\]\end{proof}
Each $e_I=e_{i_1}\wedge\cdots\wedge e_{i_k}$ is an eigenvector with with weight $\frac{1}{2}\left(\sum_{i \in I} L_i-\sum_{j \notin I} L_j\right)$,and the highest weight is  $\frac{1}{2}\left(L_1+\cdots+L_n\right)$ with eigenvector $e_1\wedge\cdots\wedge e_n$.

Whether $n$ is odd or even, we have \[L_{e_i}\circ (\mathrm{Id}|_{\bigwedge^{\text{even}}W}-\mathrm{Id}|_{\bigwedge^{\text{odd}}W})(e_1\wedge\cdots\wedge e_n)=\pm e_i\wedge(e_1\wedge\cdots\wedge e_n)=0\]
\[D_{2e_i^*}\circ (\mathrm{Id}|_{\bigwedge^{\text{even}}W}-\mathrm{Id}|_{\bigwedge^{\text{odd}} W})=\pm 2e_1\wedge\cdots e_{i-1}\wedge e_{i+1}\wedge\cdots\wedge e_n.\] Notice that this term will not contribute when pairing with $(e_1\wedge\cdots \wedge e_n)^*$.
Finally, direct calculation gives
\begin{equation}
\begin{aligned}
(1+At^{-1})(e_1\wedge\cdots e_n) ={}& [1+t^{-1}(\sum_{1 \leqslant i, j \leqslant n} a_{i j}(E_{i, j}-E_{n+j, n+i})+ \sum_{1 \leqslant i< j \leqslant n} a_{i j}(E_{i, n+j}-E_{j i n+i})\\
&+\sum_{1 \leqslant i <j\leq n} a_{i j}(E_{n+i, j}-E_{n+j, i})+\sum_{1\leq i\leq n}a_{i,2n+1}(E_{i,2n+1}-E_{2n+1,n+i}) \\
&+\sum_{1\leq i\leq n}a_{n+i,2n+1}(E_{n+i,2n+1}-E_{2n+1,i}))](e_1\wedge\cdots\wedge e_n)\\
&=(1+\frac{1}{2}t^{-1}\sum_{1\leq i\leq n}a_{ii})e_1\wedge\cdots\wedge e_n.
\end{aligned}
\end{equation}

After pairing with $(e_1\wedge\cdots\wedge e_n)^*$, we get $1+\frac{1}{2}t^{-1}\mathrm{Tr}(A_{nn})$.

In conclusion, the Hamiltonians are  $\mathrm{Tr}(A_{kk})$ and $\operatorname{Tr}\left(\bigwedge^2 A_{k, k}\right)$ for $1\leq k\leq n-1$, and $\mathrm{Tr}(A_{n,n})$ for the Spinor representation. Notice that $\operatorname{Tr}\left(\bigwedge^2 A_{1, 1}\right)=0$, so there are actually $2n-2=h^\vee-1$ Hamiltonians.

\section*{Acknowledgement}We are thankful to Yehao Zhou for many discussions with the lemmas in Section 2, to Michael Finkelberg for bringing this problem to our attention and sharing many his ideas. Special thanks go to Rui Xiong for teaching us to use Sage and many generous discussions for the computations.   Additionally, we would like to express our appreciation to Roman Travkin for the conjecture of the quantization of the Hamiltonians, and to Olivier Schiffmann for teaching us a lot of techniques of the Lie algebra representation. Finally, we would like to thank the reviewer for taking the time and effort necessary to review the manuscript. We sincerely appreciate all valuable comments and suggestions, which helped us to improve the quality of the manuscript. This work was supported by Fondation  Mathématique Jacques Hadamard (FMJH) during the author's stay in Université Paris Saclay.

\printbibliography{}

\end{document}